\documentclass[11pt]{article}
\usepackage{graphicx}

\usepackage{color}

\usepackage{amsfonts}
\usepackage{amssymb}
\usepackage{amsmath}
\usepackage{amsthm}
\usepackage{amscd}
\usepackage{pxfonts}
\usepackage{stmaryrd}
\usepackage{amsfonts}
\usepackage{bbm}
\usepackage{amscd}
\usepackage{mathrsfs}
\usepackage{latexsym,amsxtra}
\usepackage[T1]{fontenc}
\usepackage{lmodern}
\usepackage[all]{xy}
\usepackage{nicefrac,mathtools}
\usepackage{microtype}

\usepackage{enumitem}  
\usepackage{calc}
\setlist{labelindent=1pt,itemsep=.5em}
\setlist[itemize]{leftmargin=1.2cm}
\setlist[enumerate]{itemindent=0em,leftmargin=1.2cm}
\setlist[enumerate,1]{label={\upshape(\roman*)}}

\usepackage{cite}

\usepackage{hyperref}
\hypersetup{
colorlinks   = true,
citecolor    = black,
linkcolor =black,
urlcolor=black,
}

\allowdisplaybreaks[3]

\usepackage{authblk}

\makeatletter
\newcommand{\subjclass}[2][2020]{%
  \let\@oldtitle\@title%
  \gdef\@title{\@oldtitle\footnotetext{#1 \emph{Mathematics subject classification}: #2}}%
}
\newcommand{\keywords}[1]{%
  \let\@@oldtitle\@title%
  \gdef\@title{\@@oldtitle\footnotetext{\emph{Keywords}: #1}}%
}
\makeatother

\makeatletter

\newtheorem{thm}{Theorem}[section]
\newtheorem{prop}{Proposition}[section]
\newtheorem{lem}{Lemma}[section]
\newtheorem{rem}{Remark}[section]

\newtheorem{cor}{Corollary}[section]
\newtheorem{defn}{Definition}[section]

\newtheorem{ex}{Example}[section]

\numberwithin{equation}{section}

\date{}

\title{Kupershmidt operators on Hom-Malcev algebras and their deformation }
\author{Fattoum Harrathi$^{1}$, Sami Mabrouk$^{2}$\footnote{{\it Corresponding author}: Sami Mabrouk, mabrouksami00@yahoo.fr}, Othmen Ncib$^{2}$, Sergei Silvestrov$^{3}$\\
\normalsize{
$^{1}$University of Sfax, Faculty of Sciences Sfax,  BP 1171, 3038 Sfax, Tunisia. \authorcr
harrathifattoum285@gmail.com   \authorcr \vspace{0,2cm}
$^{2}$University of Gafsa, Faculty of Sciences Gafsa, 2112 Gafsa, Tunisia.\authorcr
mabrouksami00@yahoo.fr, othmen.ncib@fsgf.u-gafsa.tn \authorcr \vspace{0,2cm}
$^3$M\"{a}lardalen University,
Division of Mathematics and Physics, \authorcr
School of Education, Culture and Communication, \authorcr
Box 883, 72123 V\"{a}ster{\aa}s, Sweden. \authorcr
sergei.silvestrov@mdu.se}}

\begin{document}
\maketitle

\abstract{The main feature of Hom-algebras is that the identities defining
the structures are twisted by linear maps. The purpose of this paper is to introduce and study a Hom-type generalization of
pre-Malcev algebras, called Hom-pre-Malcev algebras. We also introduce the notion of Kupershmidt operators
of Hom-Malcev and Hom-pre-Malcev algebras and
show the connections between Hom-Malcev and Hom-pre-Malcev  algebras using Kupershmidt operators. Hom-pre-Malcev algebras generalize Hom-pre-Lie algebras to the Hom-alternative setting and fit into a bigger
framework with a close relationship with Hom-pre-alternative algebras. Finally, we establish a deformation theory of Kupershmidt operators on a Hom-Malcev algebra in consistence
with the general principles of deformation theories and introduce the notion of
Nijenhuis elements. \\*[0,3cm]
}


\newpage

\tableofcontents

\numberwithin{equation}{section}
\section{Introduction}~~

Malcev algebras were introduced in 1955 by A. I. Malcev \cite{Malcev55:anltcloops}, who called these objects Moufang-Lie algebras because of the connection with analytic Moufang loops. A Malcev algebra is a non-associative algebra $A$ with an anti-symmetric multiplication $[-,-]$ that satisfies, for all $x,y,z \in A$, the Malcev identity
\begin{equation}
\label{malcev}
J(x,y,[x,z]) = [J(x,y,z),x],
\end{equation}
where $J(x,y,z) = [[x,y],z] + [[z,x],y] + [[y,z],x]$ is the Jacobian.  In particular, Lie algebras are examples of Malcev algebras. Malcev algebras play an important role in Physics and the geometry of smooth loops. Just as the tangent
algebra of a Lie group is a Lie algebra, the tangent algebra of a locally analytic Moufang
loop is a Malcev algebra \cite{kerdman,kuzmin68:Malcevalgrepr,kuzmin71:conectMalcevanMoufangloops,Malcev55:anltcloops,nagy,sabinin}, see also \cite{gt,myung,okubo} for discussions about connections with physics.
The notion of pre-Malcev algebra as a Malcev algebraic analogue of a pre-Lie algebra was introduced in \cite{Madariaga}.
A pre-Malcev algebra is a vector space $A$ with a bilinear
multiplication $''\cdot'' $ such that the product $[x,y] = x\cdot y-y \cdot x$ endows
$A$ with the structure of a Malcev algebra, and the left multiplication operator
$L_{\cdot}(x): y \mapsto x \cdot y$ define a representation of this Malcev algebra on $A$.
In other words, the product  $x \cdot y$ satisfies  the following identities for all $x, y, z, t \in A$:
\begin{equation}\label{PM}
[y,z] \cdot (x \cdot t)+ [[x,y],z] \cdot t+ y \cdot ([x,z] \cdot t)- x \cdot (y \cdot (z \cdot t)) + z \cdot (x \cdot (y \cdot t)).
\end{equation}

The theory of Hom-algebras has been initiated in \cite{HartwigLarSil:defLiesigmaderiv, LarssonSilvJA2005:QuasiHomLieCentExt2cocyid,LarssonSilv:quasiLiealg} from Hom-Lie algebras, quasi-Hom-Lie algebras and quasi-Lie algebras, motivated by quasi-deformations of Lie algebras of vector fields, in particular q-deformations of Witt and Virasoro algebras. Hom-Lie algebras and more general quasi-Hom-Lie algebras were introduced first by Hartwig, Larsson and Silvestrov in  \cite{HartwigLarSil:defLiesigmaderiv} where a general approach to discretization of Lie algebras of vector fields using general twisted derivations ($\sigma$-deriva\-tions) and a general method for construction of deformations of Witt and Virasoro type algebras based on twisted derivations have been developed. The general quasi-Lie algebras, containing the quasi-Hom-Lie algebras and Hom-Lie algebras as subclasses, as well their graded color generalization, the color quasi-Lie algebras including color quasi-hom-Lie algebras, color hom-Lie algebras and their special subclasses the quasi-Hom-Lie superalgebras and hom-Lie superalgebras, have been first introduced in \cite{HartwigLarSil:defLiesigmaderiv,LarssonSilvJA2005:QuasiHomLieCentExt2cocyid,LarssonSilv:quasiLiealg,LSGradedquasiLiealg,LarssonSilv:quasidefsl2,SigSilv:CzechJP2006:GradedquasiLiealgWitt}.
Subsequently, various classes of Hom-Lie admissible algebras have been considered in \cite{ms:homstructure}. In particular, in \cite{ms:homstructure}, the Hom-associative algebras have been introduced and shown to be Hom-Lie admissible, that is leading to Hom-Lie algebras using commutator map as new product, and in this sense constituting a natural generalization of associative algebras as Lie admissible algebras leading to Lie algebras using commutator map. Furthermore, in \cite{ms:homstructure}, more general $G$-Hom-associative algebras including Hom-associative algebras, Hom-Vinberg algebras (Hom-left symmetric algebras), Hom-pre-Lie algebras (Hom-right symmetric algebras), and some other Hom-algebra structures, generalizing $G$-associative algebras, Vinberg and pre-Lie algebras respectively, have been introduced and shown to be Hom-Lie admissible, meaning that for these classes of Hom-algebras, the operation of taking commutator leads to Hom-Lie algebras as well. Diagram \eqref{diagramhomliealg} illustrates the relations existing between some of these
structures. (Note that Rota-Baxter operators can be replaced by the more general
Kupershmidt operators in the upper and lower rows.)
 \begin{equation}\label{diagramhomliealg}
\begin{split}
\resizebox{13cm}{!}{\xymatrix{
\ar[rr] \mbox{\bf Hom-dendriform alg $(A,\prec,\succ,\alpha)$ }\ar[d]_{\mbox{ $\ast=\prec+\succ$}}\ar[rr]^{\mbox{\quad\quad $\cdot=\prec-\succ^{op}$\quad}}
                && \mbox{\bf Hom-pre-Lie alg $(A,\cdot ,\alpha)$ }\ar[d]_{\mbox{ Commutator}}\\
\ar[rr] \mbox{\bf Hom-associative alg $(A,\ast,\alpha)$}\ar@<-1ex>[u]_{\mbox{ R-B }}\ar[rr]^{\mbox{\quad\quad Commutator\quad}}
                && \mbox{\bf Hom-Lie alg  $(A,[-,-],\alpha)$}\ar@<-1ex>[u]_{\mbox{ R-B}}}
}\end{split}
\end{equation}

Also, Hom-flexible algebras have been introduced, connections to Hom-algebra generalizations of derivations and of adjoint maps have been noticed, and some low-dimensional Hom-Lie algebras have been described.
In Hom-algebra structures, defining algebra identities are twisted by linear maps.
Since the pioneering works
\cite{HartwigLarSil:defLiesigmaderiv,LarssonSilvJA2005:QuasiHomLieCentExt2cocyid,LarssonSilv:quasiLiealg,LSGradedquasiLiealg,LarssonSilv:quasidefsl2,ms:homstructure}, Hom-algebra structures have developed in a popular broad area with increasing number of publications in various directions.
Hom-algebra structures include their classical counterparts and open new broad possibilities for deformations, extensions to Hom-algebra structures of representations, homology, cohomology and formal deformations, Hom-modules and hom-bimodules, Hom-Lie admissible Hom-coalgebras, Hom-coalgebras, Hom-bialgebras, Hom-Hopf algebras, $L$-modules, $L$-comodules and Hom-Lie quasi-bialgebras, $n$-ary generalizations of biHom-Lie algebras and biHom-associative algebras, generalized derivations, Rota-Baxter operators, Hom-dendriform color algebras, Rota-Baxter bisystems and covariant bialgebras, Rota-Baxter cosystems, coquasitriangular mixed bialgebras, coassociative Yang-Baxter pairs, coassociative Yang-Baxter equation and generalizations of Rota-Baxter systems and algebras, curved $\mathcal{O}$-operator systems and their connections with tridendriform systems and pre-Lie algebras, BiHom-algebras, BiHom-Frobenius algebras and double constructions, infinitesimal biHom-bialgebras and Hom-dendriform $D$-bialgebras, and category theory of Hom-algebras
\cite{AmmarEjbehiMakhlouf:homdeformation,
AttanLaraiedh:2020ConstrBihomalternBihomJordan,
Bakayoko:LaplacehomLiequasibialg,
Bakayoko:LmodcomodhomLiequasibialg,
BakBan:bimodrotbaxt,
BakyokoSilvestrov:HomleftsymHomdendicolorYauTwi,
BakyokoSilvestrov:MultiplicnHomLiecoloralg,
BenMakh:Hombiliform,
BenAbdeljElhamdKaygorMakhl201920GenDernBiHomLiealg,
CaenGoyv:MonHomHopf,
DassoundoSilvestrov2021:NearlyHomass,
GrMakMenPan:Bihom1,
HassanzadehShapiroSutlu:CyclichomolHomasal,
HounkonnouDassoundo:centersymalgbialg,
HounkonnouDassoundo:homcensymalgbialg,
HounkonnouHoundedjiSilvestrov:DoubleconstrbiHomFrobalg,
kms:narygenBiHomLieBiHomassalgebras2020,
Laraiedh1:2021:BimodmtchdprsBihomprepois,
LarssonSilvJA2005:QuasiHomLieCentExt2cocyid,
LarssonSilv:quasidefsl2,
LarssonSigSilvJGLTA2008:QuasiLiedefFttN,
LarssonSilvestrovGLTMPBSpr2009:GenNComplTwistDer,
MaMakhSil:CurvedOoperatorSyst,
MaMakhSil:RotaBaxbisyscovbialg,
MaMakhSil:RotaBaxCosyCoquasitriMixBial,
MaZheng:RotaBaxtMonoidalHomAlg,
MabroukNcibSilvestrov2020:GenDerRotaBaxterOpsnaryHomNambuSuperalgs,
Makhlouf2010:ParadigmnonassHomalgHomsuper,
MakhSil:HomHopf,
MakhSilv:HomDeform,
MakhSilv:HomAlgHomCoalg,
MakYau:RotaBaxterHomLieadmis,
RichardSilvestrovJA2008,
RichardSilvestrovGLTbnd2009,
SaadaouSilvestrov:lmgderivationsBiHomLiealgebras,
ShengBai:homLiebialg,
Sheng:homrep,
SilvestrovParadigmQLieQhomLie2007,
SigSilv:GLTbdSpringer2009,
SilvestrovZardeh2021:HNNextinvolmultHomLiealg,
Yau:ModuleHomalg,
Yau:HomEnv,
Yau:HomHom,
Yau:HombialgcomoduleHomalg,
Yau:HomYangBaHomLiequasitribial}.

A Hom-type generalization of Malcev algebras, called Hom-Malcev
algebras, is defined in \cite{Yau}, where connections between Hom-alternative algebras and Hom-Malcev algebras are given. We aim in this paper to introduce and study, through Rota-Baxter operators and Kupershmidt operators, the relationship between Hom-Malcev and Hom-pre-Malcev algebras generalizing, then, Malcev  algebras and pre-Malcev algebras.
The anti-commutator of a Hom-pre-Malcev algebra is a Hom-Malcev algebra and the left multiplication operators give a representation of this Hom-Malcev algebra, which is the beauty of such a structure.

The deformation of algebraic structures began with the seminal work of Gerstenhaber \cite{Gerstenhaber1, Gerstenhaber2, Gerstenhaber3, Gerstenhaber4}
 for associative
algebras and followed by its extension to Lie algebras by
Nijenhuis and Richardson~\cite{NR,NR2}. Deformations of other
algebraic structures such as pre-Lie and Malcev algebras have also been
developed in ~\cite{Bu0} and \cite{Nijenhuis} respectively. In general, deformation theory was developed
for binary quadratic operads by Balavoine~\cite{Bal}. For more
general operads we refer the reader to  \cite{KSo, LV, Fo}.
Nijenhuis operators also play an important role in  deformation
theories  due to their relationship with trivial infinitesimal
deformations. There are interesting applications of Nijenhuis operators such as constructing biHamiltonian systems to study the integrability of nonlinear evolution equations \cite{CGM,Do}.

In this paper we use Kupershmidt operators to split operations, although a generalization exists in the alternative setting in terms of bimodules: the Rota-Baxter operators
defined by Madariaga \cite{Madariaga}. Diagram \eqref{Diagramme1} summarizes the results of the present work.

In Section \ref{sec: prelimenaires}, we  summarize the definitions and some key constructions of  Hom-alternative algebras and Hom-Malcev algebras, we introduce the notion of Kupershmidt operators  of Hom-Malcev algebras that generalizes the notion of Rota-Baxter operators and we develop (dual) representation of a Hom-Malcev algebras calles also $\mathcal{O}$-operators. In Section \ref{sec:Hom-pre-Malcev  algebras}, we introduce the notion of Hom-pre-Malcev algebra, provide  some properties and define the notion of a bimodule of a Hom-pre-Malcev algebra. Moreover, we develop  some constructions theorems. We show that on one hand, a Kupershmidt operators  on a Hom-Malcev algebra gives a Hom-pre-Malcev algebra. On the other hand, a Hom-pre-Malcev algebra naturally gives a Kupershmidt operator on the sub-adjacent Hom-Malcev algebra. In Section \ref{sec:infdef}, we study deformations of Kuperchmidt operators and introduce the notion of Nijenhuis elements.

In this paper, all vector spaces are over a field $\mathbb{K}$ of characteristic $0$.
\section{Preliminaries and basics}\label{sec: prelimenaires}

 The purpose of this section is to recall some basics about Hom-alternative  algebras introduced in \cite{Makhl:HomaltHomJord, Yau}. The  definition of Hom-Malcev  algebras can  be viewed as a Hom-alternative algebra via the commutator bracket (see \cite{Yau}). We introduce  the notion of  representation of Hom-Malcev algebra. Also, we construct the dual representation of a representation of a Hom-Malcev algebra
without any additional condition. This is nontrivial. Finally, we introduce the notion of Kupershmidt operator on a Hom-Malcev algebra which is a generalisation of Rota-Baxter operators.


A Hom-algebra is a triple $(A, \mu, \alpha)$ in which $A$ is a vector space, $\mu: A^{\otimes2}\longrightarrow A$ and $\alpha: A\longrightarrow  A$ are two linear maps. Hom-algebra is said to be multiplicative if $\alpha\circ \mu = \mu \circ \alpha^{\otimes 2}.$ We denote be $End(A)$ the space of all linear map $f:A\to A$.

 \begin{defn}[\cite{Makhl:HomaltHomJord}]
A Hom-alternative algebra is a Hom-algebra $(A, \ast, \alpha)$ satisfying
for all $x, y, z \in A$,
\begin{equation}
\label{homalt}
as_{\alpha}(x,x,y)= as_{\alpha}(y, x, x)= 0,
\end{equation}
where $as_{\alpha}(x, y, z) = (x\ast y)\ast\alpha(z) -\alpha(x)\ast (y\ast z)$ is  the Hom-associator \textup{(}$\alpha$-associator\textup{)}.
\end{defn}

 \begin{defn}[\cite{Yau}]
A Hom-Malcev algebra is a Hom-algebra $(A, [-, -], \alpha)$ such that $[-, -]$ is anti-symmetric, and satisfies the Hom-Malcev identity for all $x, y, z \in A$,
\begin{equation}
\label{Hom-Malcev:Jacobiannotation}
J_{\alpha}(\alpha (x), \alpha (y), [x,z]) = [J_{\alpha}(x, y, z), \alpha^{2}(x)],
\end{equation}
where $J_{\alpha}(x, y, z)=[[x,y],\alpha (z)]+[[y,z],\alpha (x)]+[[z,x],\alpha (y)]$ is the Hom-Jacobian of $x, y, z$ \textup{(}$\alpha$-Jacobian\textup{)}.
\end{defn}

The Hom-Malcev identity \eqref{Hom-Malcev:Jacobiannotation} is equivalent to
\begin{equation}
\label{Hom-Malcev}
\begin{array}{ll}
[\alpha ([x, z]), \alpha ([y, t])] &= [[[x, y], \alpha (z)], \alpha^{2}(t)] + [[[y, z], \alpha (t)], \alpha^{2} (x)]\\
& \quad + [[[z, t], \alpha (x)], \alpha^{2} (y)] + [[[t, x], \alpha (y)], \alpha^{2} (z)].
\end{array}
\end{equation}
When $\alpha = Id_A$, we recover the Malcev algebra (see \cite{Malcev55:anltcloops}).

A Hom-Malcev algebra $(A, [- , -], \alpha)$ is said to be {\bf regular}, if $\alpha$ is inversible.  $(A, [- , -], \alpha)$ is called {\bf multiplicative} if
$\alpha([x, y]) = [\alpha(x), \alpha(y)].$

Now we introduce the notion of Hom-Malcev admissible algebra and generalize the well-known fact that alternative
algebras are Malcev admissible algebras
\begin{defn}
  A Hom-algebra $(A, [-, -], \alpha)$  is said to be a Hom-Malcev admissible algebra if, for any
elements $x, y \in A$, the bracket $[-,-] : A \times A \to A$ defined by
$$[x, y] = x \ast y - y \ast x$$
satisfies the Hom-Malcev identity.
\end{defn}

It is proved in \cite{ms:homstructure} that, every Hom-associative algebra $A$ is  Hom-Lie admissible  algebra.  Also, any alternative algebra is a Malcev admissible algebra. The following theorem shows that
Hom-alternative algebras are related to Hom-Malcev algebras via admissibility \cite{Yau}.

\begin{thm} Any Hom-alternative algebra $(A, \ast, \alpha)$ is a Hom-Malcev  admissible algebra. That is  $(A, [- , -], \alpha)$ is a Hom-Malcev algebra where $[x, y] = x\ast y - y\ast x$ for all $x, y \in A$.
\end{thm}
 Let  $(A, [-, -], \alpha)$ and  $(A', [-, -]', \alpha')$ be two Hom-Malcev algebras. A
linear map $f : A\to A' $ is said to be a morphism of Hom-Malcev algebras
if for all $ x, y\in A,$
$$[f(x)\, f(y)]' = f([x, y])~~\text{and}~~ f \circ \alpha = \alpha'\circ f.$$
\begin{defn}
Let $(A,[-, -], \alpha)$ be a Hom-Malcev algebra and let $V$ be a vector space. Then a linear map $\varrho : A \longrightarrow End(V )$ is a  representation of $(A,[-, -], \alpha)$ on $V$ with respect to $\beta \in
End(V )$ if for any $x, y, z \in A,$
\begin{eqnarray} \label{rephommalcev}
\varrho(\alpha(x))\beta &=&  \beta\varrho(x),\\
\varrho([[x, y], \alpha (z)])\beta^{2} & =&
\varrho(\alpha^{2}(x))\varrho(\alpha (y))\varrho(z)-  \varrho(\alpha^{2}(z))\varrho(\alpha (x))\varrho(y)  \nonumber \\
 & & +  \varrho(\alpha^{2}(y))\varrho([z,x]) \beta -\varrho(\alpha ([y, z]))\varrho(\alpha (x))\beta.
\label{representation H-M}
\end{eqnarray}
\end{defn}

\begin{prop} \label{semidirectprduct HomMalcev}
Let $(A,[-, -], \alpha)$ be a Hom-Malcev algebra, $(V, \beta)$ be a vector space,  and $\varrho : A \longrightarrow End(V )$ be a linear map. Then $(V, \varrho, \beta)$ is a representation of $A$ if and only if
$(A \oplus V, [-, -]_{\varrho}, \alpha + \beta)$ is a Hom-Malcev algebra, where $[-, -]_{\varrho}$ and $\alpha + \beta$ are defined for all  $x, y \in  A,\ a, b \in V$ by
\begin{eqnarray*}
[x + a, y + b]_{\varrho} &=& [x, y] + \varrho(x)a - \varrho(y)b, \\
(\alpha + \beta)(x + a) &=& \alpha(x) + \beta(a).
\end{eqnarray*}
This Hom-Malcev algebra is called the semi-direct product of $(A,[-, -], \alpha)$ and $(V, \beta)$, and denoted by
$A \ltimes _{\varrho}^{\alpha, \beta} V$ or simply $A \ltimes V$.
\end{prop}
\begin{ex}
 Let $(A,[-, -], \alpha)$ be a Hom-Malcev algebra.
Then $ad_{x} : A \rightarrow End(A)$, defined for all $x, y \in A$ by
$ad_{x}(y) = [x , y],$ is a representation of $(A, [-, -], \alpha)$  on $A$  with respect to $\alpha$, called \textbf{the adjoint representation} of $A$.
\end{ex}
\begin{thm}
Let $(V, \varrho, \beta)$ be a representation over the multiplicative Hom-Malcev algebra $(A, [-,-], \alpha)$.
Then $(V, \varrho^\alpha, \beta)$ is a representation over $(A, [-,-], \alpha)$, with
$ \varrho^\alpha(x) a=\varrho(\alpha^{2}(x))(a)$.
\end{thm}

\begin{proof}
By using the multiplicativity of $\alpha$, and then \eqref{representation H-M} for $\alpha^2(x)$, $\alpha^2(y)$ and $\alpha^2(z)$ in $(V,\varrho, \beta)$, for any  $x, y,z \in A$ and $a\in V$, we have
\begin{align*}
 \varrho^{\alpha}([[x,y],\alpha (z)])\beta^{2}(a)  = & \varrho(\alpha^{2}([[x,y],\alpha^{3} (z)]))\beta^{2}(a)\\
  = & \varrho(\alpha^{4} (x))\varrho(\alpha^{3} (y))\varrho(\alpha^{2}(z))(a) - \varrho(\alpha^{4} (z))\varrho(\alpha^{3} (x))\varrho(\alpha^{2}(y))(a) \\
   &+\varrho(\alpha^{4}(y))\varrho(\alpha^{2}([z,x]))\beta(a)-\varrho(\alpha^{3} ([y,z])\varrho(\alpha^{3} (x))\beta(a)\\
   =& \varrho^{\alpha}(\alpha^{2} (x))\varrho^{\alpha}(\alpha (y))\varrho^{\alpha}(z)(a) -\varrho^{\alpha}(\alpha^{2} (z))\varrho^{\alpha}(\alpha (x))\varrho^{\alpha}(y)(a) \\
   &  + \varrho^{\alpha}(\alpha^{2}(y))\varrho^{\alpha}([z,x])\beta(a)
   -\varrho^{\alpha}(\alpha ([y,z])\varrho^{\alpha}(\alpha (x))\beta(a).
   \qedhere
 \end{align*}
\end{proof}
\begin{thm}
Let $(A, [-, -])$ be a Malcev algebra and $\alpha: A \longrightarrow A$ be an algebra morphism. Then
the Hom-algebra $A_{\alpha} = (A, [-, -]_{\alpha} = \alpha\circ[-, -], \alpha)$ induced by $\alpha$ is a multiplicative Hom-Malcev algebra.
Moreover, assume that $(A', [-, -]')$ is another Malcev algebra, and $\alpha': A' \to A'$ is a
Malcev algebra morphism. Let $f : A \to A'$ be a Malcev algebra
morphism satisfying $f \circ \alpha = \alpha'\circ f$. Then $f : A_{\alpha} \to A'_{\alpha'}$
 is a Hom-Malcev algebra morphism.
 \label{twistingmalcevalg}
\end{thm}
The following result gives a procedure to construct representation of Hom-Malcev algebra by
a representation of Malcev algebra,  morphism and linear map.
\begin{prop} \label{prop:twistingrepresentation}
   Let $(A, [-, -])$ be a Malcev algebra, $\alpha: A\to A$ be a morphism on $A$,  $(V, \varrho)$ be a representation of $A$ and $\beta:V\to V$ be a linear map such that
 $\beta \varrho(x)(a)=\varrho(\alpha(x))\beta(a).$
 Then $(V, \widetilde{\varrho},\beta)$ is a representation of the multiplicative Hom-Malcev algebra $(A, [-, -]_\alpha,\alpha)$, where
 $$\widetilde{\varrho}(x)(a)=\varrho(\alpha(x))\beta(a),\ \forall \  \ x\in A,\ a\in V.$$
 \end{prop}
 \begin{proof}
 For any $x, y, z \in A$ and $a \in V$, by \eqref{rephommalcev}, we have
 $$\beta \widetilde{\varrho}(x)(a)=\beta(\varrho(\alpha(x))\beta(a))=\varrho(\alpha^{2}(x))\beta^{2}(a)=\widetilde{\varrho}(\alpha(x))\beta(a).$$
Combining with \eqref{representation H-M} and using the multiplicativity condition, we get
\begin{align*}
&\widetilde{\varrho}([[x, y]_\alpha, \alpha (z)]_\alpha)(\beta^{2}(a))-  \widetilde{\varrho}(\alpha^{2}(x))\widetilde{\varrho}(\alpha (y))\widetilde{\varrho}(z)(a) +\widetilde{\varrho}(\alpha^{2}(z))\widetilde{\varrho}(\alpha (x))\widetilde{\varrho}(y)(a)\\
&+ \widetilde{\varrho}(\alpha^{2}(y)) \widetilde{\varrho}([x,z]_\alpha)\beta(a) + \widetilde{\varrho}(\alpha ([y,z]_\alpha))\widetilde{\varrho}(\alpha (x))\beta(a)  \\=&
 \beta^{3}\Big(\varrho([[x, y], z]) - \varrho(x)\varrho(y)\varrho(z) + \varrho(z)\varrho(x)\varrho(y) +\varrho(y)\varrho([x,z]) +\varrho([y, z])\varrho(x)\Big)(a) = 0.
 \end{align*}
Then $(V, \widetilde{\varrho},\beta)$ is a representation of $(A, [-, -]_\alpha,\alpha)$ on $V$.
\end{proof}

Let $V^*$ denote  the dual vector space of $V$. Denote the canonical pairing between $V^{*}$ and $V$ by
$$$$
\begin{equation}
\langle\cdot ,\cdot \rangle:V^* \times V\rightarrow \mathbb{K},\quad a^*(b)= \langle a^*,b\rangle=\langle b,a^*\rangle,\;\;\forall  a^*\in V^*,b\in V.
\end{equation}
Let $(V,\varrho,\beta)$ be a representation of a Hom-Malcev algebra $(A,[-,-],\alpha)$ and $V^{*}$ be the dual of vector space $V$.

In the sequel, we always assume that $\beta$ is invertible.
Define $\varrho^*:A\longrightarrow gl(V^*)$ as usual by
$$\langle \varrho^*(x)(\xi),a\rangle=-\langle\xi,\varrho(x)(a)\rangle,\quad \forall \ x\in A,a\in V,\xi\in V^*.$$
However, in general $\varrho^*$ is not a representation of $A$ anymore. Define $\varrho^\star:A\longrightarrow gl(V^*)$ by
\begin{equation}\label{eq:new1}
 \varrho^\star(x)(\xi)=\varrho^*(\alpha(x))\big((\beta^{-2})^*(\xi)\big),\quad\forall \ x\in A,\xi\in V^*.
\end{equation}

\begin{lem}\label{lem:dualrep}
 Let $(V,\varrho,\beta)$ be a representation of a Hom-Malcev algebra $(A,[-,-],\alpha)$ and $V^{*}$ be the dual of vector space $V$, where $\beta$ is invertible. Then $\varrho^\star:A\longrightarrow gl(V^*)$ defined above by \eqref{eq:new1} is a representation of $(A,[-,-],\alpha)$ on $V^*$ with respect to $(\beta^{-1})^*$, which is called the  dual representation of $(V,\varrho,\beta)$.
\end{lem}
\begin{proof}
For all $x\in A,\xi\in V^*$,  we have
\begin{eqnarray*}
\varrho^\star(\alpha(x))((\beta^{-1})^*(\xi))=\varrho^*(\alpha^{2}(x))(\beta^{-3})^*(\xi)=(\beta^{-1})^*(\varrho^*(\alpha(x))(\beta^{-2})^*(\xi))=(\beta^{-1})^*(\varrho^\star(x)(\xi)),
\end{eqnarray*}
which implies $\varrho^\star\big{(}\alpha(x)\big{)}\circ(\beta^{-1})^*=(\beta^{-1})^*\circ\varrho^\star(x)$.

On the other hand, for all $x,y,z\in A,\xi\in V^*$ and $a\in V$, we have
\begin{align*}
&\Big\langle \varrho^{\star}([[x,y],\alpha(z)])((\beta^{-2})^*(\xi)),a\Big\rangle =\Big\langle\varrho^*(\alpha ([[x,y],\alpha(z)]))((\beta^{-4})^*(\xi)),a\Big\rangle\\
&=-\Big\langle ((\beta^{-4})^*(\xi)),\varrho(\alpha([[x,y],\alpha(z)]))(a)\Big\rangle\\
&=-\Big\langle ((\beta^{-4})^*(\xi)),\varrho(\alpha^{3}(x))(\varrho(\alpha^{2}(y))(\varrho(\alpha(z))(\beta^{-2}(a)))) -\varrho(\alpha^{3}(z))(\varrho(\alpha^{2}(x))(\varrho(\alpha(y))(\beta^{-2}(a))))\\
 &\quad+ \varrho(\alpha^{3}(y))(\varrho(\alpha([z, x]))(\beta^{-1}(a))) - \varrho(\alpha^{2}([y, z]))(\varrho(\alpha^{2}(x))(\beta^{-1}(a)))\Big\rangle\\
&=-\Big\langle ((\beta^{-6})^*(\xi)),\varrho(\alpha^{5}(x))(\varrho(\alpha^{4}(y))(\varrho(\alpha^{3}(z))(a))) - \varrho(\alpha^{5}(z))(\varrho(\alpha^{4}(x))(\varrho(\alpha^{3}(y))(a)))\\
&\quad+ \varrho(\alpha^{5}(y))(\varrho(\alpha^{3}([z, x]))(\beta(a))) - \varrho(\alpha^{4}([y, z]))(\varrho(\alpha^{4}(x))(\beta(a)))\Big\rangle\\
&=-\Big\langle -\varrho^{\ast}(\alpha^{3}(z))(\varrho^{\ast}(\alpha^{4}(y))(\varrho^{\ast}(\alpha^{5}(x))((\beta^{-6})^*(\xi)))) \\
&\quad + \varrho^{\ast}(\alpha^{3}(y))(\varrho^{\ast}(\alpha^{4}(x))(\varrho^{\ast}(\alpha^{5}(z))((\beta^{-6})^*(\xi))))\\
&\quad + \varrho^{\ast}(\alpha^{2}([z, x]))(\varrho^{\ast}(\alpha^{4}(y))(\beta^{-5})^* (\xi)) - \varrho^{\ast}(\alpha^{3}(x))(\varrho^{\ast}(\alpha^{3}([y, z]))((\beta^{-5})^*(\xi))), a\Big\rangle\\
&=-\Big\langle -\varrho^{\star}(\alpha^{2}(z))\varrho^{\star}(\alpha (y))\varrho^{\star}(x)+  \varrho^{\star}(\alpha^{2}(y))\varrho^{\star}(\alpha (x))\varrho^{\star}(z)\\
&\quad+ \varrho^{\star}(\alpha([z, x]))(\varrho^{\star}(\alpha(y))((\beta^{-1})^* (\xi))) - \varrho^{\star}(\alpha^{2}(x))(\varrho^{\star}([y, z])((\beta^{-1})^*(\xi))),a\Big\rangle,
\end{align*}
which implies that
\begin{align*}
\varrho^\star([[x,y],\alpha(z)])(\beta^{-2})^*&=\varrho^{\star}(\alpha^{2}(z))\varrho^{\star}(\alpha (y))\varrho^{\star}(x)-  \varrho^{\star}(\alpha^{2}(y))\varrho^{\star}(\alpha (x))\varrho^{\star}(z)\\
&\quad- \varrho^{\star}(\alpha([z, x]))\varrho^{\star}(\alpha(y))(\beta^{-1})^*  + \varrho^{\star}(\alpha^{2}(x))\varrho^{\star}([y, z])(\beta^{-1})^*.
\end{align*}
Therefore, $\varrho^\star$ is a representation of $(A,[-,-],\alpha)$ on $V^*$ with respect to $(\beta^{-1})^*$.
\end{proof}

\begin{lem}\label{lem:dualdual}
Let  $(V,\varrho,\beta)$ be a representation of a Hom-Malcev algebra  $(A,[-,-],\alpha)$ and let $V^{*}$ be dual of $V$, where $\beta$ is invertible. Then  $(\varrho^\star)^\star=\varrho.$
\end{lem}
\begin{proof}
For all $x\in A,a\in V,\xi\in V^*$, we have
\begin{eqnarray*}
\langle (\varrho^\star)^\star(x)(a),\xi\rangle
&=&\langle(\varrho^\star)^*(\alpha(x))(\beta^2(a)),\xi\rangle
=-\langle\beta^2(a),\varrho^\star(\alpha(x))(\xi)\rangle\\
&=&-\langle\beta^2(a),\varrho^*(\alpha^2(x))((\beta^{-2})^*(\xi))\rangle
=\langle\varrho(\alpha^2(x))(\beta^2(a)),(\beta^{-2})^*(\xi)\rangle\\
&=&\langle\varrho(x)(a),\xi\rangle,
\end{eqnarray*}
which implies that $(\varrho^\star)^\star=\varrho$.
\end{proof}

Consider the dual representation of the adjoint representation, we have
\begin{cor}\label{lem:rep}
 Let $(A,[-,-],\alpha)$ be a regular Hom-Malcev algebra and let $A^{*}$ be the dual vector space of $A$. Then $ad^\star:A\longrightarrow A^*$  defined
 for all $x\in A, \ \xi\in A^*$ by
 \begin{equation} \label{eq:adstar}
  ad^\star(x)\xi=ad^*(\alpha(x))(\alpha^{-2})^*(\xi),
 \end{equation}
 is a representation of $(A,[-,-],\alpha)$ on $A^*$ with respect to $(\alpha^{-1})^*$, called the co-adjoint representation.
\end{cor}
Using the co-adjoint representation $ad^\star$, we can obtain a semidirect product Hom-Malcev algebra
structure on $A\oplus A^*$.

\begin{cor}
Let $(A,[-,-],\alpha)$ be a regular Hom-Malcev algebra and let $A^{*}$ be the dual vector space of $A$.  Then there is a natural Hom-Malcev algebra
$(A\oplus A^*,[-,-]_s,\alpha+(\alpha^{-1})^*)$ where the Hom-Malcev bracket $[-,-]_s$ is given by
$$[x+\xi,y+\eta]=[x,y]+ad^\star(x)(\eta)-ad^\star(y)(\xi),$$
for all $x, y\in A,\ \xi , \eta\in A^*.$\end{cor}

Let $(A,[-,-],\alpha)$ be a Hom-Malcev algebra  and  $\varrho$ be a representation of $A$ on $V$ with respect to $\beta$. We say that we have a \textsf{Hom-MalcevRep} pair
and refer to it with the tuple
 $(A,[-,-],\alpha, \varrho,\beta)$.
Then, the following terminology is motivated by the notion of  Kupershmidt operator as a generalization of Rota-Baxter operator of weight $0$.
\begin{defn}
A linear map $T : V \to  A$ is called a Kupershmidt operator on a \textsf{Hom-MalcevRep} pair $(A,[-,-],\alpha, \varrho,\beta)$  if
for all $a, b \in V,$
\begin{align}\label{ophommalcev}
 \alpha\circ T =  T\circ\beta, \quad \quad [T (a), T (b)] = T \big(\varrho(T (a))b - \varrho(T (b))a\big).
 \end{align}
 \label{df:ophommalcalg}
\end{defn}
 \begin{ex}
 A Rota-Baxter operator of  weight $0$ on a Hom-Malcev algebra $A$ is
just a Kupershmidt operator on a \textsf{Hom-MalcevRep} pair $(A,[-,-],\alpha,ad)$, that is, $\mathcal{R}$ satisfies
$$ \mathcal{R}\circ\alpha = \alpha\circ \mathcal{R}, \quad \quad [\mathcal{R}(x),\mathcal{R}(y)] = \mathcal{R}([\mathcal{R}(x), y] + [x, \mathcal{R}(y)])$$
for all $x,y\in A$.
\end{ex}
\begin{ex}
If $\alpha=Id_A$ and $\beta=Id_V$, then the Definition \ref{df:ophommalcalg} coincides with the notion of Kupershmidt operators on  Malcev algebra.
\end{ex}

\begin{ex}
Let $T:V\rightarrow A$ be a Kupershmidt operator on the Malcev algebra $( A,[-,-])$ with respect $(V,\varrho)$, $\alpha: A\to A$ be a morphism on $A$ and $\beta:V\to V$ be a linear map such that
 $\alpha\circ T=T\circ \beta.$
 By Theorem \ref{twistingmalcevalg} and Proposition \ref{prop:twistingrepresentation}, for all $a,b \in V$,
\begin{align*}
[T(a),T(b)]_{\alpha}
&= \alpha[T(a),T(b)]=\alpha\big(T(\varrho(T(a))b-\varrho(T(b))a\big),\\
T\big(\tilde{\varrho}(T(a))b-\tilde{\varrho}(T(b))a\big)
&= T\big((\varrho(\alpha(T(a)))\beta(b)-\varrho(\alpha(T(b)))\beta(a)\big).
\end{align*}
 Clearly, it follows that the map $T:V\rightarrow A$ is a Kupershmidt operator on the \textsf{Hom-MalcevRep} pair  $( A,[-,-]_\alpha,\alpha,\tilde{\varrho}, \beta)$.
\end{ex}
The following proposition gives a characterization of a Kupershmidt operator $T$ in terms of a Hom-Malcev subalgebra structure on the graph of $T$.

\begin{prop}
A map $T:V\rightarrow A$ is  a Kupershmidt operator on a  \textsf{Hom-MalcevRep} pair $(A,[-,-],\alpha,\varrho)$ if and only if the graph $\mathrm{Gr}(T)=\{(T(a),a)|~a\in V\}$ of the map $T$
is a  Hom-Malcev subalgebra of the semi-direct product Hom-Malcev algebra $(A\oplus V,[-,-]_{\rho},\alpha+\beta)$, defined in Proposition \ref{semidirectprduct HomMalcev}.
\end{prop}

It is well-known that Kupershmidt operators on Malcev algebras can be characterized in terms of the Nijenhuis operators. In the next result, we characterize Kupershmidt operators on Hom-Malcev algebras in terms of the Nijenhuis operators. Let us first recall  that a linear map $N:A\rightarrow A$ is called a Nijenhuis operator on the Hom-Malcev algebra $(A,[-,-],\alpha)$ if $N$ satisfies  $\alpha\circ N=N\circ \alpha$ and the equation
$$[N(x),N(y)]=N\big([N(x),y]-[N(y),x]-N([x,y])\big), $$
for all  $x,y\in A$. Then, we have the following characterization of Kupershmidt operators on Hom-Malcev algebras.
\begin{prop}
A map $T:V\rightarrow A$ is a Kupershmidt operator on $(A,[-,-],\alpha,\varrho,\beta)$ if and only if the operator
$$N_T=\begin{bmatrix}
   0 & T \\
    0  & 0
\end{bmatrix}:A\oplus V\rightarrow A\oplus V $$
is a Nijenhuis operator on the semi-direct product Hom-Malcev algebra $(A\oplus V,[-,-]_{\varrho},\alpha+\beta)$.
\end{prop}

\begin{proof}
First, it is obvious that $N_T\circ(\alpha + \beta) = (\alpha + \beta) \circ N_T$ if and only if $T \circ \beta = \alpha\circ T$ .\\
 Then, let us consider the following expressions, where we use definition of the map $N_T$ and the bracket $[-,-]_{\varrho}$,
\begin{align}\label{Char2: eq1}
& [N_T(x+a),N_T(y+b)]_{\varrho}= [T(a)+0,T(b)+0]_{\varrho}=[T(a),T(b)], \\
\label{Char2: eq2}
\nonumber
&
N_T\big([N_T(x+a),y+b]_{\varrho}-[N_T(y+b),x+a]_{\varrho}-N_T([x+a,y+b]_{\varrho})\big)\\\nonumber
& =N_T \big(([T(a),y]+\varrho(T(a))b)-([T(b),x]+\varrho(T(b))a)-(0+T(\varrho(x)b-\varrho(y)a))\big)\\
& = T(\varrho(T(a))b-\varrho(T(b))a)
\end{align}
for all $x,y\in A,$ $a,b\in V$.
By \eqref{Char2: eq1} and \eqref{Char2: eq2}, it is clear that the condition
$$[N_T(x+a),N_T(y+b)]_{\rho}=N_T\big([N_T(x+a),y+b]_{\rho}-[N_T(y+b),x+a]_{\rho}-N_T([x+a,y+b]_{\rho})\big)$$
is equivalent to the condition
$$[T(a),T(b)]=T(\varrho(T(a))b-\varrho(T(b))a),$$
for all $x,y\in A,$ $a,b\in V$.
\end{proof}

 \section{Hom-pre-Malcev  algebras}\label{sec:Hom-pre-Malcev  algebras}
In this section, we generalize the notion of pre-Malcev algebras introduced in \cite{Madariaga} to the Hom
case and study the relationships with Hom-Malcev algebras in terms of Kupershmidt operators of Hom-Malcev algebras and Hom-pre-alternative algebras, and we also construct some examples of finite dimensional  Hom-Malcev algebras via Rota-Baxter operators. Moreover, we characterize the notion of  representation of Hom-pre-Malcev algebras via semidirect product  and provide some key constructions.
\subsection{Definition and basic properties}\label{subsec:Definition and basic properties}
\begin{defn}
A  $\textbf{Hom-pre-Malcev algebra}$ is  a  Hom-algebra $(A, \cdot, \alpha)$ satisfying, for any $x,y,z,t \in A$ and $[x,y]=x\cdot y-y\cdot x,$ the identity
\begin{equation}\label{HPM}
\begin{split}
  &[\alpha(y),\alpha(z)] \cdot \alpha(x \cdot t) + [[x , y] , \alpha(z)] \cdot \alpha^{2}(t) \\
&\qquad + \alpha^{2}(y) \cdot ([x , z] \cdot \alpha (t)) - \alpha^{2}(x) \cdot (\alpha (y)  \cdot (z \cdot t)) + \alpha^{2}(z) \cdot (\alpha (x) \cdot (y \cdot t))=0.\end{split}
\end{equation}
\end{defn}

The identity \eqref{HPM} is equivalent to
$HPM(x, y, z, t) = 0$, where for all $x,y,z,t\in A,$
\begin{equation}\label{HPMexpanded}
\begin{split}
HPM(x, y, z, t) &=  \alpha(y \cdot z) \cdot \alpha(x \cdot t) - \alpha(z \cdot y) \cdot \alpha(x \cdot t)\\
&\quad+ ((x \cdot y) \cdot \alpha(z)) \cdot \alpha^{2}(t) - ((y \cdot x) \cdot \alpha (z)) \cdot \alpha^{2}(t)\\
 &\quad- (\alpha (z) \cdot (x \cdot y)) \cdot \alpha^{2}(t) + (\alpha (z) \cdot (y \cdot x)) \cdot \alpha^{2}(t)\\
&\quad+ \alpha^{2}(y) \cdot ((x \cdot z) \cdot \alpha (t)) - \alpha^{2}(y) \cdot ((z \cdot x) \cdot \alpha (t))\\
 &\quad- \alpha^{2}(x) \cdot (\alpha (y)  \cdot (z \cdot t)) + \alpha^{2}(z) \cdot (\alpha (x) \cdot (y \cdot t)).
\end{split}
\end{equation}
Observe that when $\alpha = Id_A$, the Hom-pre-Malcev algebra $(A, \cdot, \alpha)$ reduces to the pre-Malcev algebra.
\begin{defn}
Let  $(A, \cdot, \alpha)$  be a Hom-pre-Malcev algebras. It is called as follows:
\begin{description}
\item[(1)] Multiplicative Hom-pre-Malcev algebra if $\alpha$ satisfies
\begin{equation}\label{Multiplicativityalpha}
    \alpha(x \cdot y) = \alpha(x) \cdot \alpha(y),\quad  \forall x, y \in A,
\end{equation}
\item[(2)] Regular  Hom-pre-Malcev algebra if $\alpha$ is inversible,
\item[(3)] Involutive  Hom-pre-Malcev algebra if $\alpha^2=Id$.
\end{description}
\end{defn}
Hom-pre-Malcev algebras generalize Hom-pre-Lie algebras. A Hom-pre-Lie algebra
 is a vector space $A$ with a bilinear product "$\bullet$" and a linear map $\alpha$ satisfying the Hom-pre-Lie identity for all $x, y, z \in A$,
$$HPL(x, y, z) = as_\alpha(x, y, z) - as_\alpha(y, x, z) = 0,$$
where $as_\alpha(x, y, z)=(x\bullet y)\bullet\alpha(z)-\alpha(x)\bullet(y\bullet z)$ is the Hom-associator. Note that
\begin{align*}
HPM(x, y, z, t) &= HPL([x, y],\alpha (z), \alpha (t)) - HPL([y, x], \alpha (z),  \alpha (t)) \\
& \quad+ HPL(\alpha (x), \alpha (y), [z, t])+ HPL(\alpha (y), \alpha (z), [x, t]) \\
& \quad- [\alpha^{2}(z), HPL(x, y, t)] + [\alpha^{2}(y), HPL(x, z, t)].
\end{align*}
So every Hom-pre-Lie algebra is a Hom-pre-Malcev algebra.

\begin{defn}
 Let  $(A, \cdot, \alpha)$ and  $(A', \cdot', \alpha')$ be two Hom-pre-Malcev algebras. A
linear map $f : A\to A' $ is called
\begin{description}
  \item[(1)] a weak morphism of Hom-pre-Malcev algebras if it satisfies, for all $x, y\in A$,
$$f(x)\cdot' f(y) = f(x\cdot y),$$
  \item[(2)]
 a morphism of Hom-pre-Malcev algebras if $f$ is a weak morphism and $f \circ \alpha = \alpha'\circ f.$
\end{description}
\end{defn}

The following theorem provides a procedure to construct  pre-Malcev algebra by a regular Hom-pre-Malcev algebra.
 \begin{thm}
 Let $\mathcal{A}=(A, \cdot)$ be a regular Hom-pre-Malcev algebra. Then $\mathcal{A}_{\alpha^{-1}} =(A, \cdot_{\alpha^{-1}})$ is a pre-Malcev algebra with $x \cdot_{\alpha^{-1}} y = \alpha^{-1} (x \cdot y).$
 \end{thm}
 \begin{proof}
 We just show that $(A, \cdot_{\alpha^{-1}})$ satisfies the
identity \eqref{PM} while $(A, \cdot,\alpha)$
satisfies the  identity \eqref{HPM}. Indeed,
\begin{align*}
[y,z]_{\alpha^{-1}}\cdot_{\alpha^{-1}} x \cdot_{\alpha^{-1}} t &= \alpha^{-3}\big(\alpha([y,z])\cdot\alpha(x\cdot t)\big), \\
[[x,y]_{\alpha^{-1}},z]_{\alpha^{-1}} \cdot_{\alpha^{-1}}t&= \alpha^{-3}([[x,y], \alpha(z)]\cdot \alpha^{2}(t)),\\
y \cdot_{\alpha^{-1}} ([x,z]_{\alpha^{-1}}\cdot_{\alpha} t) &= \alpha^{-3}(\alpha^{2}(y)\cdot ([x,z]\cdot \alpha(t))), \\
  x \cdot_{\alpha^{-1}} (y  \cdot_{\alpha^{-1}} (z \cdot_{\alpha^{-1}} t))&=\alpha^{-3}(\alpha^{2}(x)\cdot (\alpha(y) \cdot (z \cdot t))), \\
z \cdot_{\alpha^{-1}} (x \cdot_{\alpha^{-1}} (y \cdot_{\alpha^{-1}} t)) &= \alpha^{-3}(\alpha^{2}(z)\cdot (\alpha(x)\cdot(y\cdot t))). \qedhere
\end{align*}
\end{proof}
In particular, the theorem is valid when $\alpha=Id.$

In the following starting from a Hom-pre-Malcev algebra and a Hom-pre-Malcev algebra endomorphism, we contructt a new Hom-pre-Malcev algebra. We say that is obtained by twisting principale or composition method.
\begin{prop}
  Let $(A, \cdot, \alpha)$ be a Hom-pre-Malcev algebra, and let $\gamma : A\rightarrow A$ be a Hom-pre-Malcev algebra endomorphism. Then
  $A_\gamma=(A, \cdot_\gamma, \gamma\alpha)$, where $x\cdot_\gamma y=\gamma(x\cdot y)$, is a Hom-pre-Malcev algebra.

Moreover, suppose that $(\mathcal{A}', \cdot')$ is a pre-Malcev algebra and $\alpha': A\rightarrow A'$ is a pre-Malcev algebra endomorphism. If $f : A\rightarrow A'$ is a pre-Malcev algebra morphism that satisfies $f\circ\gamma=\alpha'\circ f$, then
$$f :  (A, \cdot_\gamma, \beta\alpha)\rightarrow
(\mathcal{A}', \cdot_{\alpha'}', \alpha')$$
is a morphism of Hom-pre-Malcev algebras.
\end{prop}
\begin{proof}
For all $x, y, z\in A$,
\begin{align*}
&\gamma\alpha([y,z]_{\gamma})\cdot_{\gamma} \gamma\alpha(x \cdot_{\gamma} t)+ [[x,y]_{\gamma},\gamma\alpha(z)]_{\gamma} \cdot_{\gamma}(\gamma\alpha)^{2}(t)\\
&=\gamma^{3}\big(\alpha([y,z])\cdot \alpha(x\cdot t)\big)+\gamma^{3}([[x,y], \alpha(z)]\cdot \alpha^{2 }(t))\\
&=\gamma^{3}\big(\alpha(^{2}(y) \cdot ([z,x]\cdot \alpha (t))+ \alpha^{2}(x) \cdot (\alpha (y)  \cdot (z \cdot t)) - \alpha^{2}(z) \cdot (\alpha (x) \cdot (y \cdot t))\big)\\
&=(\gamma\alpha)^{2}(y) \cdot_{\gamma} ([z,x]_{\gamma}\cdot_{\gamma} \gamma\alpha (t))+(\gamma\alpha)^{2}(x) \cdot_{\gamma} (\gamma\alpha (y)  \cdot_{\gamma} (z \cdot_{\gamma} t)) - (\gamma\alpha)^{2}(z) \cdot_{\gamma} (\gamma\alpha (x) \cdot_{\gamma} (y \cdot_{\gamma} t)).
  \end{align*}
For the second part,
$
f(x\cdot_\gamma y)=f(\gamma(x\cdot y))=\alpha'( f(x\cdot y))
=\alpha'(f(x)\cdot'f(y))=f(x)\cdot'_{\alpha'} f(y).
$
\end{proof}
\begin{ex}
Let $(A, \cdot)$  be a pre-Malcev algebra and  $\alpha : A \to A$ be a morphism of $A$.
Then $A_{\alpha} = (A, \cdot_{\alpha}, \alpha)$ is a multiplicative Hom-pre-Malcev algebra.
\end{ex}
the following Proposition shows that  any Hom-pre-Malcev algebra is Hom-Malcev admissible.
\begin{prop} \label{prop:HompreMalcevHomMalcevadmis}
 Let $(A, \cdot, \alpha)$ be a Hom-pre-Malcev algebra.
 The commutator given, for all $x,\ y\in A$, by
 \begin{align}\label{commutator}
 [x, y] = x \cdot y - y \cdot x,
 \end{align}
  defines a Hom-Malcev algebra structure on $A$.
 \end{prop}
\begin{proof}
We show that the commutator \eqref{commutator} satisfies
the identity \eqref{Hom-Malcev}. For $x, y, z, t \in A$,
\begin{align*}
&[\alpha ([x, z]), \alpha([y, t])] - [[[x, y], \alpha (z)], \alpha^{2} (t)] - [[[y, z], \alpha(t)], \alpha^{2}(x)] \\
&\quad - [[[z, t], \alpha (x)], \alpha^{2}(y)]- [[[t, x], \alpha (y)], \alpha^{2}(z)]\\
&=\alpha([x,z]) \cdot \alpha(y \cdot t) - \alpha ([x,z]) \cdot \alpha(t \cdot y) - \alpha([y, t]) \cdot \alpha(x \cdot z)\\
  &\quad+ \alpha([y, t]) \cdot \alpha(z \cdot x)-  [[x, y], \alpha(z)] \cdot \alpha^{2}(t)+ \alpha^{2}(t) \cdot ([x, y] \cdot \alpha (z)) \\
&\quad- \alpha^{2}(t) \cdot (\alpha (z) \cdot (x \cdot y))+ \alpha^{2}(t) \cdot (\alpha(z) \cdot (y \cdot x))- [[y, z], \alpha (t)] \cdot \alpha^{2}(x)\\
&\quad+ \alpha^{2}(x) \cdot ([y, z] \cdot \alpha(t)) - \alpha^{2}(x) \cdot (\alpha(t) \cdot (y \cdot z))+ \alpha^{2}(x) \cdot (\alpha(t) \cdot (z \cdot y))\\
 &\quad- [[z, t], \alpha(x)] \cdot \alpha^{2}(y) + \alpha^{2}(y) \cdot ([z, t] \cdot \alpha(x))- \alpha^{2}(y) \cdot (\alpha(x) \cdot (z \cdot t))\\
 & \quad+ \alpha^{2}(y) \cdot (\alpha(x) \cdot (t \cdot z))- [[t, x], \alpha(y)] \cdot \alpha^{2}(z)+ \alpha^{2}(z) \cdot ([t, x] \cdot \alpha(y))\\
 &\quad- \alpha^{2}(z) \cdot (\alpha(y) \cdot (t \cdot x))+ \alpha^{2}(z) \cdot (\alpha(y) \cdot (x \cdot t))\\
&= HPM(x, t, y, z)  + HPM(y, x, z, t) + HPM(z, y, t, x)  + HPM(t, z, x, y) = 0.
\qedhere   \end{align*}
 \end{proof}

\begin{defn}
The Hom-Malcev algebra structure in Proposition \ref{prop:HompreMalcevHomMalcevadmis} is called the associated Hom-Malcev algebra of the Hom-pre-Malcev algebra $(A,\cdot, \alpha)$, and the  Hom-pre-Malcev algebra $(A, \cdot, \alpha)$ is called a compatible Hom-pre-Malcev algebra on the Hom-Malcev algebra $(A, [-, -], \alpha)$.
\end{defn}
The next result says that every multiplicative Hom-pre-Malcev algebra gives rise to an
infinite sequence of multiplicative Hom-pre-Malcev algebras and multiplicative Hom-Malcev algebras.
\begin{cor}
Let $(A, \cdot, \alpha)$ be a multiplicative Hom-pre-Malcev algebra. Then,
\begin{enumerate}
\item For $n\geq 0$,
$A^n=(A, \cdot^{(n)}, \alpha^{2^n})$ where
$x \cdot^{(n)} y=\alpha^{2^n-1}(x\cdot y),$
is a multiplicative Hom-pre-Malcev algebra, called the $n\mathrm{th}$ derived multiplicative Hom-pre-Malcev algebra.
\item For $n\geq0$,
$A^n=(A, [-,-]^{(n)}, \alpha^{2^n})$, where
$$[x,y]^{(n)}=\alpha^{2^n-1}(x\cdot y-y\cdot x)$$
is a multiplicative Hom-Malcev algebra, called the $n\mathrm{th}$ derived multiplicative Hom-Malcev algebra.
\end{enumerate}
\end{cor}
A Hom-pre-Malcev algebra can be viewed as a Hom-Malcev algebra whose
operation decomposes into two compatible pieces.

Examples of Hom-pre-Malcev algebras can be constructed from Hom-Malcev algebras with Kupershmidt operators. Let $(A,[-,-],\alpha)$ be a Hom-Malcev algebra  and $T : V \to A$ be a Kupershmidt operator of \textsf{Hom-MalcevRep} pair  $(A,[-,-],\alpha,\varrho,\beta)$. Define the product $"\cdot"$ by
\begin{equation}\label{eq:hommalcTohompremalc}
a\cdot b= \varrho(T(a))b, \ \forall\ a, \ b \in V.
\end{equation}

\begin{prop}\label{hommalcev==>hompremalcev}
With the above notations $(V,\cdot, \beta)$ is a Hom-pre-Malcev algebra, and there exists an associated Hom-Malcev algebra structure on $V$ given by \eqref{commutator} and $T$
is a homomorphism of Hom-Malcev algebras. Furthermore, $T(V) = \{T(a),\ a \in V \}\subset A$ is a Hom-Malcev
subalgebra of $(A, [-, -], \alpha)$ and $T (V)$ is a Hom-pre-Malcev algebra structure given,
for all $a, b \in V,$ by $T(a)\cdot T(b) = T(a\cdot b). $
Moreover, the corresponding associated Hom-Malcev algebra structure on  $T (V)$ given by \eqref{commutator} is just a
Hom-Malcev subalgebra  structure of $(A, [-, -], \alpha)$, and $T$ is a morphism of Hom-pre-Malcev algebras.
\end{prop}
\begin{proof}
By the identity of Kupershmidt operator \eqref{ophommalcev},
$$T([a,b]_{T})=T(a\cdot b-b\cdot a)=[T(a), T(b)].$$
Thanks to \eqref{representation H-M}, for any $a,b,c,d\in V$,
\begin{align*}
&[\beta(b),\beta(c)]_{T } \cdot \beta(a \cdot d) + [[a ,b]_{T } , \beta(c)]_{T } \cdot \beta^{2}(d) + \beta^{2}(b) \cdot ([a , c]_{T} \cdot \beta (d))\\
 &\quad- \beta^{2}(a) \cdot (\beta (b)  \cdot (c \cdot d)) + \beta^{2}(c) \cdot (\beta (a) \cdot (b \cdot d))\\
 &= \varrho(\alpha([T(b), T(c)]))\varrho(\alpha(T(a)))\beta(d)+\varrho([[T(a), T(b)], \alpha(T(c))])\beta^{2}(d)\\
  &\quad+\varrho(\alpha^{2}(T(b)))\varrho([T(a), T(c)])\beta(d)-\varrho(\alpha^{2}(T(a)))\varrho(\alpha(T(b)))\varrho(T(c))d\\
  &\quad+\varrho(\alpha^{2}(T(c)))\varrho(\alpha(T(a)))\varrho(T(b))d= 0.
\end{align*}
So, $(V,\cdot,\beta)$ is a Hom-pre-Malcev algebra. The other statements follow immediately.
\end{proof}
An obvious consequence of Proposition \ref{hommalcev==>hompremalcev} is the construction of a Hom-pre-Malcev algebra in terms of a Rota-Baxter operator of weight zero of a Hom-Malcev
algebra.
\begin{cor}
Suppose that $\mathcal{R}: A \longrightarrow A$ is a Rota-Baxter operator on a Hom-Malcev algebra
$(A,[-,-],\alpha)$. Then $(A, \cdot, \alpha)$ is a Hom-pre-Malcev algebra, where for all $x, y \in A,$
$$x \cdot y = [\mathcal{R}(x), y].$$
\end{cor}
\begin{ex}
\label{ex:4dmalcev}
There is a four-dimensional Malcev algebra $(A,[-,-])$  with multiplication table {\rm \cite[Example 3.1]{Sagle}} for a basis $\{e_1,e_2,e_3,e_4\}$,
\begin{center}
\begin{tabular}{c|cccc}
$[-,-]$ & $e_1$ & $e_2$ & $e_3$ & $e_4$ \\ \hline
$e_1$ & $0$ & $-e_2$ & $-e_3$ & $e_4$ \\
$e_2$ & $e_2$ & $0$ & $2e_4$ & $0$ \\
$e_3$ & $e_3$ & $-2e_4$ & $0$ & $0$ \\
$e_4$ & $-e_4$ & $0$ & $0$ & $0$ \\
\end{tabular}
\end{center}
Let $\mathcal{R}$ be the operator defined, with respect to the basis $\{e_1,e_2,e_3, e_{4}\}$, by
\begin{align*}
    \mathcal{R}(e_1)= e_1 + \frac{a_{4}}{2}e_4,\ \mathcal{R}(e_2)=\lambda_1 e_3,\ \mathcal{R} (e_3)= \mathcal{R}(e_4) = 0,
\end{align*}
where $a_4$ and $\lambda_1$ are parameters in $\mathbb{K}$. By a direct computation, we can verify that $\mathcal{R}$ is a Rota-Baxter operator on $A$.

Now, using  the previous corollary,   there is a pre-Malcev algebra structure
on $A$ with the multiplication $"\cdot"$ given for all $x,y \in A$ by
$x\cdot y= [\mathcal{R}(x),  y]$, that is
$$
\begin{array}{c|cccc}
  \cdot  & e_1 & e_2 & e_3 &  e_4 \\   \hline
  e_1& \frac{-a_4}{2} e_4 & -e_2  & - e_3 & e_4 \\
  e_2 & \lambda_1 e_3 & -2\lambda_{1} e_4 &  0 & 0\\
  e_3 & 0 & 0  & 0 & 0\\
  e_4 & 0 & 0  & 0 & 0
  \end{array}$$
Using suitable algebra morphism $\alpha$, we can twist the Malcev algebra $A$ into  multiplicative Hom-Malcev algebras.
With a bit of computation, one can check that one class of algebra morphisms $\alpha \colon A \to A$ is given by
\[
\alpha(e_1) = e_1 + a_4e_4,\quad
\alpha(e_2) = -e_2 + b_3e_3,\quad
\alpha(e_3) = -e_3,\quad
\alpha(e_4) = -e_4,\quad
\]
where $a_4$ and $b_3$ are arbitrary scalars in $\mathbb{K}$.

There is a multiplicative Hom-Malcev algebra $A_{\alpha} = (A,[-,-]_{\alpha} = \alpha\circ[-,-],\alpha)$
with multiplication table
\begin{center}
\begin{tabular}{c|cccc}
$[-,-]_{\alpha}$ & $e_1$ & $e_2$ & $e_3$ & $e_4$ \\ \hline
$e_1$ & $0$ & $-\alpha(e_2)$ & $e_3$ & $-e_4$ \\
$e_2$ & $\alpha(e_2)$ & $0$ & $ -2e_4$ & $0$ \\
$e_3$ & $-e_3$ & $2e_4$ & $0$ & $0$ \\
$e_4$ & $e_4$ & $0$ & $0$ & $0$ \\
\end{tabular}
\end{center}
Then we can check that $\mathcal{R}$ is a Rota-Baxter operator on Hom-Malcev algebra $A_{\alpha}$. Therefore, there exists a  multiplicative Hom-pre-Malcev algebra $ A_\alpha=(A,\cdot_{\alpha} ,\alpha)$  where  the  multiplication $"\cdot_{\alpha}"$ is  given by
$x\cdot_{\alpha} y = \alpha([\mathcal{R}(x), y]),~~
\forall \  x,y \in A,$
  that is
$$
\begin{array}{c|cccc}
  \cdot_{\alpha}  & e_1 & e_2 & e_3 &  e_4 \\ \hline
  e_1& \frac{a_4}{2} e_4 & -\alpha(e_2)  & e_3 & -e_4 \\
  e_2 & -\lambda_1 e_3 & 2 \lambda_1 e_4 &  0 & 0\\
  e_3 & 0 & 0  & 0 & 0\\
  e_4 &  0 & 0  & 0 & 0
  \end{array}$$
\end{ex}
\begin{ex}
\label{ex:5dmalcev}
There is a five-dimensional  Malcev algebra $(A,[-,-])$ with multiplication table {\rm \cite[Example 3.4]{Sagle}} for a basis $\{e_1,e_2,e_3,e_4,e_5\}$,
\begin{center}
\begin{tabular}{c|ccccc}
$[-,-]$ & $e_1$ & $e_2$ & $e_3$ & $e_4$ & $e_5$ \\ \hline
$e_1$ & $0$ & $0$ & $0$ & $e_2$ & $0$ \\
$e_2$ & $0$ & $0$ & $0$ & $0$ & $e_3$ \\
$e_3$ & $0$ & $0$ & $0$ & $0$ & $0$ \\
$e_4$ & $-e_2$ & $0$ & $0$ & $0$ & $0$ \\
$e_5$ & $0$ & $-e_3$ & $0$ & $0$ & $0$
\end{tabular}
\end{center}
Let $\mathcal{R}$ be the operator defined, with respect to the basis $\{e_1,e_2,e_3, e_{4}, e_{5}\}$, by
\begin{align*}
    \mathcal{R}(e_1)= e_1 + a_{4}e_4 + a_{5}e_{5},\ \mathcal{R}(e_2)=b e_3,\  \mathcal{R}(e_3)= \mathcal{R}(e_5) = 0, \mathcal{R}(e_4)= \frac{-b}{a_{5}}e_{2},
\end{align*}
where $b$, $a_4$ are parameters in $\mathbb{K}$ and $a_5\in \mathbb{K}\setminus \{0\}$.
By a direct computation, we can verify that $\mathcal{R}$ is a Rota-Baxter operator on $A$.

Now, using  the previous corollary,   there is a pre-Malcev algebra structure $"\cdot"$
on $A$  given for all $x,y \in A$ by
$$x\cdot y= [\mathcal{R}(x),  y], $$ that is
$$
\begin{array}{c|ccccc}
  \cdot  & e_1 & e_2 & e_3 &  e_4 & e_5\\  \hline
  e_1& -a_4 e_2 & -a_5e_3  & 0 & e_2 & 0 \\
  e_2 & 0  & 0 &  0 & 0 & 0\\
  e_3 & 0 & 0  & 0 & 0 & 0 \\
  e_4 & 0 & 0  & 0 & 0  & \frac{-b}{a_5}e_3\\
  e_5 & 0 & 0  & 0 & 0 & 0
  \end{array}$$
Using suitable algebra morphism $\alpha$ given by
\[
\alpha(e_1) = e_1 ,\quad
\alpha(e_2) = e_2, \quad
\alpha(e_3) = e_3, \quad
\alpha(e_4) = \lambda_2e_3 + e_4,\quad
\alpha(e_5) = \frac{a_4}{a_5}\lambda_2e_3 + e_5,
\] where $\lambda_2$, $a_4$ are parameters in $\mathbb{K}$ and $a_5\in \mathbb{K}\setminus \{0\}$,
we can twist the Malcev algebra $A$ into  multiplicative Hom-Malcev algebra $A_{\alpha} = (A,[-,-]_{\alpha} = \alpha\circ[-,-],\alpha)$
with multiplication table
\begin{center}
\begin{tabular}{c|ccccc}
$[-,-]_{\alpha}$ & $e_1$ & $e_2$ & $e_3$ & $e_4$ &$e_5$ \\ \hline
$e_1$ & $0$ & $0$ & $0$ & $e_2$ & $0$\\
$e_2$ & $0$ & $0$ & $ 0$ & $0$ & $e_3$ \\
$e_3$ & $0$ & $0$ & $0$ & $0$ & $0$ \\
$e_4$ & $-e_2$ & $0$ & $0$ & $0$ & $0$\\
$e_5$ &$0$ & $-e_3$ & $0$ & $0$ & $0$
\end{tabular}
\end{center}
Then we can check that $\mathcal{R}$ is a Rota-Baxter operator on $A$. Therefore there exists a  multiplicative Hom-pre-Malcev algebras $A_\alpha=(A,\cdot_{\alpha} ,\alpha)$ with a multiplication $"\cdot_{\alpha}"$ given by
$x\cdot_{\alpha} y = \alpha([\mathcal{R}(x), y]),~~
\forall \  x,y \in A,$
 that is
$$
\begin{array}{c|ccccc}
   \cdot_{\alpha}  & e_1 & e_2 & e_3 &  e_4 & e_5\\         \hline
  e_1& -a_4 e_2 & -a_5e_3  & 0 & e_2 & 0 \\
  e_2 & 0  & 0 &  0 & 0 & 0\\
  e_3 & 0 & 0  & 0 & 0 & 0 \\
  e_4 & 0 & 0  & 0  & 0 & \frac{-b}{a_5}e_3\\
 e_5 & 0 & 0  & 0 & 0 & 0
  \end{array}$$
\end{ex}
\begin{prop}\label{pro:nsc}
   Let $(A, [-,-],\alpha)$ be a Hom-Malcev algebra. Then there exists  a compatible Hom-pre-Malcev algebra structure on $A$ if and only if there is an invertible Kupershmidt operator on $(A, [-,-],\alpha,\varrho,\beta)$.
\end{prop}
\begin{proof}
Let $(A,\cdot,\alpha)$ be a Hom-pre-Malcev algebra and $(A,[-,-],\alpha)$ be the associated Malcev algebra.  Then the identity map $id: A \to A$ is an invertible Kupershmidt operator on the Hom-MalcevRep pair  $(A,[-,-],\alpha,ad)$.

Conversely, suppose that there exists an invertible Kupershmidt operator $T$  of the Hom-MalcevRep pair  $(A,[-,-],\alpha,\varrho,\beta)$. Then, using Proposition \ref{hommalcev==>hompremalcev}, there is a Hom-pre-Malcev algebra structure on $T(V)=A$ given for all $a, b\in V$ by
\begin{equation*}
   T(a)\cdot T(b)=T(\varrho(T(a))b).
\end{equation*}
If we set $x=T(a)$ and $y=T(b)$, then we get
\begin{equation*}
   x\cdot y=T(\varrho(x)T^{-1}(y)).
\end{equation*}
This is compatible Hom-pre-Malcev algebra structure  on $(A,[-,-],\alpha)$. Indeed,
\begin{align*}
 x\cdot y-y\cdot x&= T(\varrho(x)T^{-1}(y) - \varrho(y)T^{-1}(x) )\\
&= [TT^{-1}(x),TT^{-1}(y)]=[x,y].
\qedhere
\end{align*}
\end{proof}

 Let $(A,[-,-],\alpha)$ be an any Hom-algebra and   $\omega\in\wedge^2A^*$. Recall that $\omega$ is a symplectic structure on $A$ if it satisfies
\begin{align}\label{eq:2-cocycle}
    &\omega(\alpha(x),\alpha(y)) =\omega(x,y),\quad \quad  \displaystyle\circlearrowleft_{x,y,z}\omega([x,y],\alpha(z)) =0.
\end{align}
A Hom-Malcev algebra
$(A,[-,-], \alpha)$ with a symplectic form is called a symplectic Hom-Malcev algebra.

\begin{cor}
Let $(A,[-,-],\alpha )$ be a regular Hom-Malcev algebra and  $\omega$ be a symplectic structure on  $(A,[-,-],\alpha)$.
Then there is a compatible Hom-pre-Malcev algebra structure on $A$ given by
$$
 \omega(x\cdot y,\alpha(z))=\omega(\alpha(y),[z,x] ).
$$
\end{cor}

 \begin{proof}
 Define the linear map $T : A^{*}\to A$ by $\langle T(x), y\rangle = w(x, y)$.    Since $(A,[-,-],\alpha )$ is a regular Hom-Malcev algebra, $(A^*,ad^\star,(\alpha^{-1})^*)$ is a representation of $A$. By the fact that $w$ is $\alpha$-symmetric and using  \eqref{eq:2-cocycle}, we obtain that $T$ is an invertible  Kupershmidt operator on the  Hom-MalcevRep pair $(A,[-,-],\alpha,ad^\star,(\alpha^{-1})^*)$. By Proposition~\ref{pro:nsc}, there exists a compatible Hom-pre-Malcev algebra
structure given by
\begin{equation*}
x\cdot y = T^{-1}(ad^{\star}(x)T(y)).
\end{equation*}
Hence,
\begin{align*}
w(x\cdot y,\alpha(z))=&\langle T(x\cdot y),\alpha(z)\rangle=\langle ad^{\star}(x)T(y),\alpha(z)\rangle=-\langle T(y),[\alpha(x), \alpha(z)]\rangle\\
&=-w(y, [\alpha(x), \alpha(z)] )=-w(y, \alpha([x,z]))=-w(\alpha(y), \alpha^{2}([x,z] ))\\
&=w(\alpha(y),[z, x]).
\qedhere
\end{align*}
 \end{proof}
Hom-pre-Malcev algebras are related to Hom-pre-alternative algebras analogously
to how Hom-pre-Lie algebras are related to Hom-dendriform algebras \cite{MakhloufHomdemdoformRotaBaxterHomalg2011}.
\begin{defn}[\cite{Q.Sun}]
$\textbf{A Hom-pre-alternative algebra}$ is a quadruple $(A,\prec,\succ,\alpha)$, where $A$ is a vector space,
$\prec,\succ: A \otimes A \longrightarrow A$
are bilinear maps and $\alpha\in gl(A)$  satisfying
for all $x, y, z \in A$ and  $x \ast y = x \prec y + x\succ y$,
\begin{align}
 &(x\succ y) \prec \alpha(z) - \alpha(x)\succ(y \prec z) + (y \prec x) \prec\alpha(z) - \alpha(y) \prec (x \ast z) = 0,\\
&(x\succ y)\prec\alpha(z) - \alpha(x)\succ(y \prec z) + (x \ast z)\succ \alpha(y) - \alpha(x)\succ(z\succ y) = 0,\\
&(x\ast y) \succ \alpha(z) - \alpha(x)\succ(y \succ z) + (y \ast x) \succ \alpha(z) - \alpha(y) \succ (x \succ z) = 0,\\
&(x\prec y)\prec\alpha(z) - \alpha(x)\prec(y \ast z) + (x \prec z)\prec \alpha(y) - \alpha(x)\prec(z\ast y) = 0.
\end{align}
\end{defn}
\begin{prop}
Let $(A,\prec,\succ, \alpha)$ be a Hom-pre-alternative algebra. Then $(A,\ast, \alpha)$ is a Hom-alternative algebra
\end{prop}
Now, we consider a Hom-alternative algebra $(A,\ast,\alpha)$, a vector space $V$ and a linear map $\beta:V\to V$. Recall that, a bimodule of $A$ with respect to $\beta$ is given by linear maps $\mathfrak{l},\mathfrak{r}:A\to gl(V)$ satisfying the following conditions:
\begin{align}
\label{rephomalt1}\mathfrak{l}(x^2)\beta &=\mathfrak{l}(\alpha(x))\mathfrak{l}(x),\\
\label{rephomalt2}\mathfrak{r}(x^2)\beta &=\mathfrak{r}(\alpha(x))\mathfrak{r}(x),\\
\label{rephomalt3}\mathfrak{r}(\alpha(y))\mathfrak{l}(x)-\mathfrak{l}(\alpha(x))\mathfrak{r}(y)&=\mathfrak{r}(x\ast y)\beta-\mathfrak{r}(\alpha(y))\mathfrak{r}(x),\\
\label{rephomalt4} \mathfrak{l}(y\ast x)\beta-\mathfrak{l}(\alpha(y))\mathfrak{l}(x)&=\mathfrak{l}(\alpha(y))\mathfrak{r}(x)-\mathfrak{r}(\alpha(x))\mathfrak{l}(y).
\end{align}
\begin{defn}
A Kupershmidt operator of Hom-alternative algebra $(A,\ast,\alpha)$ with respect to the bimodule $(V,\mathfrak{l},\mathfrak{r},\beta)$ is a linear map $T:V\to A$ such that,
for all $a, b \in V$,
\begin{equation}
\label{O-ophomalternative} \alpha\circ T =  T\circ\beta ~~\text{and}~~T (a)\ast T (b) = T \big(\mathfrak{l}(T (a))b + \mathfrak{r}(T (b))a\big).
 \end{equation}
 \end{defn}
 \begin{rem}
Rota-Baxter operator of weight $0$ on a Hom-alternative algebra $(A,\ast, \alpha)$ is
just a Kupershmidt operator associated to the bimodule $(A,L, R,\alpha)$, where $L$ and $R$ are the left and right
multiplication operators corresponding to the multiplication $\ast$.
 \end{rem}
\begin{prop}
With the above notations, the triplet $(V,\mathfrak{l}-\mathfrak{r},\beta)$ defines a representation of the Hom-Malcev admissible algebra $(A,[-,-],\alpha)$, and $T$ is a Kupershmidt operator of $(A,[-,-],\alpha)$ with respect to $(V,\mathfrak{l}-\mathfrak{r},\beta)$.
\end{prop}
\begin{proof}
First, note that  $(V,\mathfrak{l}, \mathfrak{r},\beta)$ is a bimodule of a Hom-alternative algebra $A$ if and only if the direct sum $(A \oplus V,\star,\alpha+\beta)$ of vector spaces is a Hom-alternative algebra (the semi-direct product) by defining multiplication in $A\oplus V$ by
$$(x+a)\star (y+b)= x\ast y+\mathfrak{l}(x)b+ \mathfrak{r}(y)a,\ \ \forall \  x,y\in A, a,b \in V.$$
Next, for its associated Hom-Malcev admissible algebra $(A \oplus V, \overbrace{[-, -]}, \alpha+\beta)$,
\begin{align*}
 \overbrace{[x+a, y+b]}=&(x+a)\star (y+b) -(y+b)\star (x+a)\\
=& x\ast y+ \mathfrak{l}(x)b+ \mathfrak{r}(y)a - y\ast x -\mathfrak{l}(y)a - \mathfrak{r}(x)b \\
=&[x, y] + (\mathfrak{l}-\mathfrak{r})(x)b - (\mathfrak{l} -\mathfrak{r})(y)a.
\end{align*}
According to Proposition \ref{semidirectprduct HomMalcev}, $(V,\mathfrak{l}-\mathfrak{r},\beta)$ is a representation of $(A,[-, -],\alpha)$.
Moreover, $T$ is a Kupershmidt operator of $(A,[-,-],\alpha)$ with respect to $(V,\mathfrak{l}-\mathfrak{r},\beta)$  since
\begin{align*}
[T(a),T(b)]=&T(a)\ast T(b)-T(b)\ast T(a)\\
=&T \big(\mathfrak{l}(T (a))b + \mathfrak{r}(T (b))a\big)-T \big(\mathfrak{l}(T (b))a + \mathfrak{r}(T (a))b\big)\\
=&T \big((\mathfrak{l}-\mathfrak{r})(T (a))b - (\mathfrak{l}-\mathfrak{r})(T (b))a\big).
\qedhere \end{align*}
\end{proof}
\begin{thm}\label{Hom-pre-altToHom-Pre-Malcev}
    Let $T:V\to A$ be a Kupershmidt operator of Hom-alternative algebra $(A,\ast,\alpha)$ with respect to the bimodule $(V,\mathfrak{l},\mathfrak{r},\beta)$. Then $(V,\prec,\succ, \beta)$ be a Hom-pre-alternative algebra, where for all $a,b\in V$,
\begin{equation}
  \label{homalt==>prehomalt} a\succ b= \mathfrak{l}(T (a))b \ \ \text{and}\ \ a\prec b= \mathfrak{r}(T (b))a.
\end{equation}
 Moreover, if $(V,\cdot,\beta)$ is the   multiplicative Hom-pre-Malcev  algebra associated to the  Hom-Malcev admissible algebra $(A,[-,-],\alpha)$ on the representation $(V,\mathfrak{l}-\mathfrak{r},\beta)$, then $a\cdot b=a\succ b-b\prec a$.

\end{thm}

\begin{proof} For any $a,b,c\in V$, using \eqref{rephomalt3} and \eqref{O-ophomalternative} yields
\begin{align*}
&(a\succ b) \prec \beta(c) -\beta(a)\succ(b \prec c) + (b \prec a) \prec\beta(c)-\beta(b) \prec (a \ast c)\\
&\quad =\mathfrak{r}(T (\beta(c)))\mathfrak{l}(T (a))b-\mathfrak{l}(T(\beta(a)))\mathfrak{r}(T(c))b \\
&\quad\quad  +\mathfrak{r}(T(\beta(c)))\mathfrak{r}(T (a))b-\mathfrak{r}(T(a \ast c))\beta(b)\\
&\quad = \mathfrak{r}(\alpha(T(c)))\mathfrak{l}(T (a))b-\mathfrak{l}(\alpha(T(a)))\mathfrak{r}(T(c))b \\
&\quad \quad +\mathfrak{r}(\alpha(T(c)))\mathfrak{r}(T(a))b-\mathfrak{r}(T(a \ast c))\beta(b)=0.
\end{align*}
The other identities for $(V,\prec,\succ,\beta)$ being a Hom-pre-alternative algebras can be verified
similarly.
Moreover, using   \eqref{hommalcev==>hompremalcev} and  \eqref{homalt==>prehomalt},
\begin{equation*}
a\cdot b=(\mathfrak{l}-\mathfrak{r})(T(a))b=\mathfrak{l}(T(a))b-\mathfrak{r}(T(a))b=a\succ b-b\prec a.
\qedhere \end{equation*}
 \end{proof}
\begin{cor}\label{homalt==>homprealt}
Let $(A, \ast , \alpha)$ be a Hom-alternative algebra and $\mathcal{R}:A\rightarrow A$ be a Rota-Baxter operator
of weight $0$ such that $ \mathcal{R}\alpha =\alpha  \mathcal{R}$. If multiplications
$\prec $ and $\succ $ on $A$ are defined for all $x, y\in A$ by
$x\prec y = x\ast \mathcal{R}(y)$ and $x\succ y = \mathcal{R}(x)\ast y$,
then $(A, \prec , \succ , \alpha)$ is a Hom-pre-alternative algebra.

Moreover, if $(A,\cdot,\alpha)$ be the  multiplicative Hom-pre-Malcev  algebra associated to the  Hom-Malcev admissible algebra $(A,[-,-],\alpha)$, then $x\cdot y=x\succ y-y\prec x$.
\end{cor}
Moreover, Hom-Malcev
algebras, Hom-alternative algebras, Hom-pre-Malcev algebras and
Hom-pre-alternative algebras are closely related as follows (in the sense of commutative diagram of categories):
\begin{equation}\label{Diagramme1}
\begin{split}
\resizebox{13cm}{!}{\xymatrix{
\ar[rr] \mbox{\bf Hom-pre-alt alg $(A,\prec,\succ,\alpha)$ }\ar[d]_{\mbox{ $\ast=\prec+\succ$}}\ar[rr]^{\mbox{\quad\quad $\cdot=\prec-\succ$\quad\quad }}
                && \mbox{\bf Hom-pre-Malcev alg $(A,\cdot ,\alpha)$ }\ar[d]_{\mbox{ Commutator}}\\
\ar[rr] \mbox{\bf Hom-alt alg $(A,\ast,\alpha)$}\ar@<-1ex>[u]_{\mbox{ R-B }}\ar[rr]^{\mbox{ Commutator\quad\quad}}
                && \mbox{\bf Hom-Malcev alg  $(A,[-,-],\alpha)$}\ar@<-1ex>[u]_{\mbox{ R-B}}}
}\end{split}
\end{equation}


\subsection{Bimodules and Kupershmidt operators of Hom-pre-Malcev algebras} \label{subsec:bimodules}


In this subsection, we introduce and study bimodules of Hom-pre-Malcev algebras and give some constructions.
\begin{defn}\label{representation HPM}
 Let $(A, \cdot, \alpha)$ be a Hom-pre-Malcev algebra and $V$ be a vector space. Let
  $\ell, r: A \longrightarrow  gl(V)$ be two linear maps and $\beta \in gl(V)$. Then $(V, \ell,  r, \beta)$ is called a bimodule of $A$  if the
following conditions hold for any $x,y,z \in A$:
\begin{align}
 &\beta \ell(x)=\ell(\alpha(x)) \beta,\quad \beta r(x)=r(\alpha(x)) \beta,\label{rep1}\\
\begin{split}
& r(\alpha^{2}(x))\varrho(\alpha (y))\varrho(z)- r(\alpha (z)\cdot(y\cdot x))\beta^{2} + \ell(\alpha^{2}(y))r(z\cdot x)\beta \\
&\quad \quad+ \ell(\alpha ([y,z]))r(\alpha (x))\beta - \ell(\alpha^{2}(z))r(\alpha (x))\varrho(y) = 0,
\end{split}
\label{rep2}\\
\begin{split}
 &\ell(\alpha^{2} (y))\ell(\alpha (z))r(x)-r(\alpha^{2}(x))\varrho(\alpha (y))\varrho(z) - \ell(\alpha^{2}(z))r(y\cdot x)\beta \\
 &\quad \quad- r(\alpha (z\cdot x))\varrho(\alpha (y))\beta+ r([z,y]\cdot\alpha (x))\beta^{2} = 0,
 \end{split}
 \label{rep3}\\
 \begin{split}
   &r(\alpha (y)\cdot(z\cdot x))\beta^{2}+r(\alpha^{2} (x))\varrho([y,z])\beta-\ell(\alpha^{2} (y))\ell(\alpha (z)) r(x) \\
    &\quad \quad+r(\alpha (y\cdot x))\varrho(\alpha (z))\beta +\ell(\alpha^{2}(z))r(\alpha (x))\varrho(y)= 0,
    \end{split}
    \label{rep4}\\
     \begin{split}
     &\ell([[x,y],\alpha (z)])\beta^{2}- \ell(\alpha^{2} (x))\ell(\alpha (y))\ell(z) + \ell(\alpha^{2} (z))\ell(\alpha (x))\ell(y) \\
   &\quad \quad+\ell(\alpha ([y,z])\ell(\alpha (x))\beta+ \ell(\alpha^{2}(y))\ell([x,z])\beta = 0,
     \end{split}
   \label{rep5}
 \end{align}
 where $\varrho(x)=\ell(x)-r(x)$ and $[x,y]= x\cdot y - y\cdot x$.
\end{defn}

Now, define a linear operation $\cdot_{\ltimes} : \otimes^{2}(A \oplus V ) \longrightarrow (A \oplus V )$ by
$$(x + a) \cdot_{\ltimes} (y + b) = x \cdot y + \ell(x)(b) + r(y)(a), ~~\forall \  x, y \in A, a, b \in V,$$
and a linear map $\alpha + \beta : A \oplus V \longrightarrow A \oplus V$ by
$$(\alpha + \beta)(x + a) = \alpha(x) + \beta(a), ~~\forall \  x \in A, a \in V.$$

\begin{prop}\label{semidirectproduct hompreMalcev}
With the above notations, $(A\oplus V, \cdot_{\ltimes}, \alpha+ \beta)$ is a  Hom-pre-Malcev algebra, which is
denoted by $A \ltimes_{(\ell,~r)}^{\alpha, \beta}V$ or simply $A \ltimes V$ and called the semi-direct product of the  Hom-pre-Malcev algebra $(A, \cdot, \alpha)$ and
the bimodule $(V, \ell, r, \beta)$.
\end{prop}
\begin{proof}
For any $x, y, z, t \in A$ and  $a, b, c, d \in V$,
\begin{align*}
&\big((\alpha + \beta)[y+b,z+c]_{\varrho}\big)\cdot_{\ltimes} (\alpha + \beta) \big((x+a)\cdot_{\ltimes}(t + d)\big) \\
&=\alpha([y,z])\cdot\alpha (x\cdot t) + \ell(\alpha ([y,z]))\big(\ell(\alpha(x))\beta(d) + r(\alpha(t))\beta(a)\big)\\
&\quad\quad + r(\alpha(x\cdot t))\big(\varrho(\alpha(y))\beta(c) -\varrho(\alpha(z))\beta(b)\big),
\\
&[[x + a,y + b]_{\varrho},(\alpha+\beta)(z+c)]_{\varrho}\cdot_{\ltimes}(\alpha^{2}+\beta^{2})(t+ d)\\
&=[[x,y],\alpha(z)]\cdot\alpha^{2}(t) + \ell([[x,y],\alpha(z)])\beta^{2}(d)\\
& \quad\quad+ r(\alpha^{2}(t))\big(\varrho([x,y])\beta(c)-\varrho(\alpha(z))(\varrho(x)b -\varrho(y)a)\big),
\\
&(\alpha^{2}+\beta^{2})(y+b)\cdot_{\ltimes}\big([x +a,z+c]_{\varrho}\cdot_{\ltimes}(\alpha+\beta)(t+ d)\big)\\
&=\alpha^{2}(y)\cdot ([x,z]\cdot\alpha(t)) + \ell(\alpha^{2}(y))\big(\ell([x,z])\beta(d) + r(\alpha(t))(\varrho(x)c-\varrho(z)a)\big)\\
 &\quad\quad+ r\big([x,z]\cdot\alpha(t)\big)\beta^{2}(b),
\\
&(\alpha^{2}+\beta^{2})(x+a)\cdot_{\ltimes}\big((\alpha+\beta)(y+b)\cdot_{\ltimes}((z +c) \cdot_{\ltimes}(t+d))\big)\\
&=\alpha^{2}(x)\cdot (\alpha(y)\cdot(z\cdot t)) + \ell(\alpha^{2}(x))\big(\ell(\alpha(y))(\ell(z)d +r(t)c) +r(z\cdot t)\beta(b)\big)\\
&\quad\quad+ r\big(\alpha(y)\cdot(z\cdot t)\big)\beta^{2}(a),
\\
&(\alpha^{2}+\beta^{2})(z+c)\cdot_{\ltimes}\big((\alpha+\beta)(x+a)\cdot_{\ltimes}((y +b) \cdot_{\ltimes}(t+d))\big)\\
&=\alpha^{2}(z)\cdot (\alpha(x)\cdot(y\cdot t)) + \ell(\alpha^{2}(z))\big(\ell(\alpha(x))(\ell(y)d +r(t)b) +r(y\cdot t)\beta(a)\big)\\
& \quad\quad+ r\big(\alpha(x)\cdot(y\cdot t)\big)\beta^{2}(c).
\end{align*}
Hence $(A \oplus V, \cdot _{\ltimes}, \alpha + \beta)$ is a Hom-pre-Malcev algebra if and only if
$(V, \ell, r, \beta)$ is a bimodule of $(A,\cdot, \alpha)$.
\end{proof}
\begin{prop}\label{rephompremalcev==rephommalcev}
  Let $(V,\ell,r,\beta)$ be a   bimodule of a Hom-pre-Malcev algebra $(A,\cdot,\alpha)$  and $(A,[-, -],\alpha)$
  be its associated Hom-Malcev algebra. Then,
  \begin{enumerate}
    \item \label{item1:prop:rephompremalcev==rephommalcev}
    $(V,\ell,\beta)$ is a representation of $(A, [-, -],\alpha)$,
    \item \label{item2:prop:rephompremalcev==rephommalcev}
    $(V,\ell - r,\beta)$ is a representation of $(A, [-, -],\alpha)$.
  \end{enumerate}
\end{prop}
\begin{proof}
\ref{item1:prop:rephompremalcev==rephommalcev}. The statement \ref{item1:prop:rephompremalcev==rephommalcev} follows immediately from \eqref{rep5}.\\
\ref{item2:prop:rephompremalcev==rephommalcev}.  By Proposition \ref{semidirectproduct hompreMalcev},
$A\ltimes_{\ell,r}^{\alpha, \beta} V$ is a Hom-pre-Malcev algebra. For its associated Hom-Malcev
 algebra $(A \oplus V, \overbrace{[-, -]}, \alpha+\beta)$,
\begin{align*}
 \overbrace{[x+a, y+b]}& =(x+a)\cdot_{\ltimes} (y+b) - (y+b)\cdot_{\ltimes} (x+a) \\
&= x \cdot y+ \ell(x)b+ r(y)a - y \cdot x - \ell(y)a - r(x)b \\
&= [x, y] + (\ell-r)(x)b - (\ell -r)(y)a.
\end{align*}
By Proposition \ref{semidirectprduct HomMalcev}, $(V,\ell-r,\beta)$ is a representation of $(A,[-, -],\alpha)$.
\end{proof}
\begin{cor}
  Let $(V,\ell,r,\beta)$ be a   bimodule of a Hom-pre-Malcev algebra $(A,\cdot,\alpha)$,  and let $(A,[-, -],\alpha)$
  be its associated Hom-Malcev algebra. Let $V^{*}$ be the dual of vector space $V$, where $\beta$ is inversible, then $(V^{*}, \ell^{\star}-r^{\star},(\beta^{-1})^{*})$ is a representation of the associated  Hom-Malcev
algebra $(A,[-, -], \alpha)$.
\end{cor}
\begin{proof}
   It follows from Proposition \ref{rephompremalcev==rephommalcev}  and Lemma \ref{lem:dualrep}.
\end{proof}

If $(A, \cdot, \alpha)$ is a Hom-pre-Malcev algebra and $(A, [-, -], \alpha)$ is the associated  Hom-Malcev algebra, then $(A, L_\cdot, \alpha)$ is a representation of $(A, [-, -], \alpha)$ , where $L_\cdot$ is the left operation of $(A, \cdot, \alpha)$ given by $L(x)(y)=x\cdot y$.
  \begin{prop}
  Let $(A,\cdot,\alpha)$ be a Hom-algebra.
  Then $(A,\cdot,\alpha)$ is a Hom-pre-Malcev algebra if and only if $(A, [-, -],\alpha)$ defined by  \eqref{commutator}
is a Hom-Malcev algebra and $(A,L_{\cdot},\alpha)$ is a representation of $(A, \cdot,\alpha)$.
\end{prop}
    \begin{proof}
       It follows from the definition of Hom-Malcev algebra and representation of  Hom-Malcev algebra. Then, for any $x,y,z,t\in A$,
    \begin{align*}
 &\alpha([y, z]) \cdot \alpha(x \cdot t) + [[x, y], \alpha(z)]\cdot \alpha^{2}(t) + \alpha^{2}(y)\cdot ([x, z] \cdot \alpha(t))\\& - \alpha^{2}(x)\cdot (\alpha(y)\cdot (z\cdot t)) + \alpha^{2}(z)\cdot (\alpha(x)\cdot (y\cdot t))\\
 &=  \Big( L_{\cdot}(\alpha[y,z])L_{\cdot}(\alpha(x))\alpha +  L_{\cdot}([[x,y],\alpha(z)])\alpha^{2} +  L_{\cdot}(\alpha^{2}(y)) L_{\cdot}([x,z])\alpha\\
  &- L_{\cdot}(\alpha^{2}(x))L_{\cdot}(\alpha(y))L_{\cdot}(z) + L_{\cdot}(\alpha^{2}(z))L_{\cdot}(\alpha(x))L_{\cdot}(x) \Big) (t) = 0.
\qedhere \end{align*}
    \end{proof}
As in \cite{Bai2,K2}, we rephrase the definition of Kupershmidt operator in terms of Hom-pre-Malcev
algebras as follows.
   \begin{defn}\label{o-ophpm}
Let $(A, \cdot, \alpha)$ be a Hom-pre-Malcev algebra and $(V,\ell,r,\beta)$ be a bimodule. A linear map $T : V \to  A $ is called a Kupershmidt operator associated to $(V,\ell,r,\beta)$
 if $T$ satisfies
 \begin{align}
 T\circ \beta &= \alpha \circ T,\\
 T (a) \cdot T (b)&= T \big(\ell(T (a))b + r(T (b))a\big), \quad\forall \  a, b \in V.
\end{align}
\end{defn}
\begin{rem}
Let $T$ is a Kupershmidt operator of a Hom-pre-Malcev algebra $(A, \cdot, \alpha)$ associated to $(V, \ell, r, \beta)$.
Then $T$ is a Kupershmidt operator of its associated Hom-MalcevRep pair  $(A, [-, -], \alpha,\ell - r, \beta)$.
\end{rem}
\begin{proof}
By Proposition \ref{rephompremalcev==rephommalcev}, for all $a, b \in V,$
\begin{align*}
[T(a), T(b)]& = T(a) \cdot T(b) - T(b) \cdot T(a)\\
& = T (\ell(T (a))b + r(T (b))a) - T (\ell(T (b))a + r(T (a))b)\\
& = T ((\ell- r)(T (a))b - (\ell -r)(T (b))a).
\qedhere \end{align*}
\end{proof}
\begin{thm}
  Let $(V, \ell, r, \beta)$ be a bimodule over the multiplicative Hom-pre-Malcev-algebra
$(A, \cdot, \alpha)$, and let $A=(A, [-,-], \alpha)$ be the associated Hom-Malcev algebra.
Then both
$(V, \ell^\alpha, \beta)$ and  $(V, \ell^\alpha- r^\alpha, \beta)$ are representations over $(A, [-,-], \alpha)$, where
\begin{align*}
& \ell^\alpha=\ell\circ(\alpha^2\otimes Id) \quad \text{and}\quad r^\alpha=r\circ(\alpha^2\otimes Id).
\end{align*}
\end{thm}
\begin{proof}
We  prove that $\ell^\alpha$ satisfy   \eqref{representation H-M}.
For any elements $x, y,z \in A$ and $a\in V$,
\begin{align*}
   \ell^{\alpha}([[x,y],\alpha (z)])\beta^{2}(a) &\stackrel{\eqref{Multiplicativityalpha}}{=}   \ell([[\alpha^{2}(x),\alpha^{2}(y)],\alpha^{3} (z)])\beta^{2}(a)\\
   & \stackrel{\eqref{rep5}}{=}   \ell(\alpha^{4} (x))\ell(\alpha^{3} (y))\ell(\alpha^{2}(z))(a) - \ell(\alpha^{4} (z))\ell(\alpha^{3} (x))\ell(\alpha^{2}(y))(a) \\
   & \quad \quad - \ell(\alpha^{3} ([y,z])\ell(\alpha^{3} (x))\beta(a)- \ell(\alpha^{4}(y))\ell(\alpha^{2}([x,z]))\beta(a)\\
  & = \ell^{\alpha}(\alpha^{2} (x))\ell^{\alpha}(\alpha (y))\ell^{\alpha}(z)(a) -\ell^{\alpha}(\alpha^{2} (z))\ell^{\alpha}(\alpha (x))\ell^{\alpha}(y)(a) \\
   & \quad \quad -  \ell^{\alpha}(\alpha ([y,z])\ell^{\alpha}(\alpha (x))\beta(a)+ \ell^{\alpha}(\alpha^{2}(y))\ell^{\alpha}([x,z])\beta(a).
   \qedhere
 \end{align*}

\end{proof}
\section{Infinitesimal deformations of a Kupershmidt operator of Hom-Malcev algebras}
\label{sec:infdef}~~
In \cite{HuLiuSheng, Sun}, the authors  studied the Kupershmidt-(dual-)Nijenhuis structures on a  Lie  and an alternative  algebra
with a representation. Recently, Sami Mabrouk in \cite{Mabrouk} studies the infinitesimal deformations of a
 Kupershmidt operator of Malcev algebras. Inspired by these works, we consider  the infinitesimal deformations of a
 Kupershmidt operator of Hom-Malcev algebras. In particular, we introduce the notion of a
Nijenhuis element associated to  a Kupershmidt operator, which gives
rise to a trivial infinitesimal deformation of the
 Kupershmidt operator. Their relationship with the infinitesimal deformations of
the associated Hom-pre-Malcev algebra is also studied.
    \begin{defn}\label{morphism of O-operators}
      Let $T$ and $T'$ be   two Kupershmidt operators on a  \textsf{Hom-MalcevRep} pair\\  $( A,[-,-],\alpha,\varrho,\beta)$. A  morphism from $T'$ to $T$ consists of a  Hom-Malcev algebra homomorphism  $\varphi_ A: A\longrightarrow A$ and a linear map $\varphi_V:V\longrightarrow V$ such that
      \begin{eqnarray}
        T\circ \phi_V&=&\phi_ A\circ T',\label{defi:isocon1}\\
        \varphi_V\circ \beta  &=&\beta\circ \varphi_V,\label{defi:isocon2}\\
        \varphi_V\varrho(x)(a)&=&\varrho(\varphi_ A(x))(\varphi_V(a)),\quad\forall \  x\in A, a\in V.\label{defi:isocon3}
      \end{eqnarray}
      In particular, if both $\varphi_ A$ and $\varphi_V$ are  invertible,  $(\varphi_ A,\varphi_V)$ is called an  isomorphism  from $T'$ to $T$.
    \label{defi:isoO}
    \end{defn}

\begin{prop}
Let $T$ and $T'$ be two  Kupershmidt operators on a \textsf{Hom-MalcevRep} pair  $( A,[-,-],\alpha,\varrho,\beta)$ and $(\varphi_ A,\varphi_V)$   a morphism from $T'$ to $T$. Then $\varphi_V$ is a morphism  of Hom-pre-Malcev algebras from $(V,\cdot')$ to $(V,\cdot)$ defined by  \eqref{eq:hommalcTohompremalc}.
\end{prop}
\begin{proof}
According to  \eqref{defi:isocon1}-\eqref{defi:isocon3}, for any  $a,b\in V$, we
have
\begin{align*}
\varphi_V(a\cdot' b)&=\varphi_V
\varrho(T'(a))(b)=\varrho(\varphi_ A(T'(a)))(\varphi_V(b))\\
&=\varrho(T(\varphi_V(a)))(\varphi_V(v))
=\varphi_V(a)\cdot \varphi_V(b).
\qedhere
\end{align*}

\end{proof}
   \begin{defn}
    Let $T$   be a Kupershmidt operator  on a \textsf{Hom-MalcevRep} pair  $( A,[-,-],\alpha,\varrho,\beta)$ and $\mathcal{T}:V\longrightarrow A$ a linear map. If  $T_t=T+t\mathcal{T}$ is still a Kupershmidt operator  on the  \textsf{Hom-MalcevRep} pair  $( A,[-,-],\alpha,\varrho,\beta)$ for all $t$, we say that $\mathcal{T}$ generates a one-parameter infinitesimal deformation of the  Kupershmidt operator $T$.
    \end{defn}
By direct computation, we can check that $T_t=T+t\mathcal{T}$ is a one-parameter
infinitesimal deformation of a Kupershmidt operator $T$ if and only if for any $a,b\in V$,
{\small \begin{align}\label{commuting condition}
\mathcal{T}\circ \beta &=\alpha\circ \mathcal{T},\\
[T(a),\mathcal{T} (b)]+[\mathcal{T}, (a),T(b)]&=T(\varrho(\mathcal{T} (a))(b)-\varrho(\mathcal{T} (b))(a))+\mathcal{T}(\varrho(T(a))(b)-\varrho(T(b))(a)),\label{eq:deform1}\\
[\mathcal{T}(a),\mathcal{T} (b)]&=\mathcal{T}(\varrho(\mathcal{T} (a))(b)-\varrho(\mathcal{T}(b) )(a)).
\label{eq:deform2}
\end{align}}
Note that  \eqref{commuting condition} and \eqref{eq:deform2} mean that $\mathcal{T}$ is a Kupershmidt operator on the \textsf{Hom-MalcevRep} pair  $( A,[-,-],\alpha,\varrho,\beta)$.

Now turning to a Hom-pre-Malcev algebra $(V,\cdot,\beta)$ given in   \eqref{eq:hommalcTohompremalc}, let  $\psi:\otimes^2V\longrightarrow V$ be a linear map. If for any $t\in\mathbb{K}$, the multiplication $''\cdot_t''$ defined by
$$
a\cdot_t b:=a\cdot b+t\psi(a,b), \;\forall \  a,b\in V,
$$
also gives a Hom-pre-Malcev algebra structure, we say that $\psi$ generates a  one-parameter infinitesimal deformation of the Hom-pre-Malcev algebra $(V,\cdot,\beta)$.

The two types of infinitesimal deformations are related as follows.

\begin{prop}
 If $\mathcal{T}$ generates a one-parameter infinitesimal deformation of a Kupershmidt operator $T$ on a \textsf{Hom-MalcevRep} pair  $( A,[-,-],\alpha,\varrho,\beta)$, then the product $\psi_\mathcal{T}$ on $V$ defined by
   $$
   \psi_\mathcal{T}(a,b):=\varrho(\mathcal{T} (a))(b),\quad\forall \  a,b\in V,
   $$
generates a one-parameter infinitesimal deformation of the associated Hom-pre-Malcev algebra $(V,\cdot,\beta)$.
\end{prop}

\begin{proof}
Denote by $''\cdot_t''$ the corresponding Hom-pre-Malcev algebra structure associated to the Kupershmidt operator $T+t\mathcal{T}$. Then we have
$$
a\cdot_t b=\varrho((T+t\mathcal{T})(a))(b)=\varrho(T(a))(b)+t\varrho(\mathcal{T}
(a))(b)=a\cdot b+t\psi_\mathcal{T} (a,b), \forall \  a,b\in V,
$$
which implies that $\psi_\mathcal{T}$ generates a one-parameter infinitesimal deformation of $(V,\cdot,\beta)$.
\end{proof}

\begin{cor}
 If $\mathcal{T}$ generates a one-parameter infinitesimal deformation of a Kupershmidt operator $T$ on a  \textsf{Hom-MalcevRep} pair  $( A,[-,-],\alpha,\varrho,\beta)$. Then the product $ \psi_\mathcal{T}$ on $V$ defined by
   $$
    \omega_\mathcal{T}(a,b):=\varrho(\mathcal{T} (a))(b)-\varrho(\mathcal{T} (b))(a),\quad\forall \  a,b\in V,
   $$
generates a one-parameter infinitesimal
deformation of the sub-adjacent  Hom-Malcev algebra $(V,[-,-]_T,\beta)$
of the associated Hom-pre-Malcev algebra $(V,\cdot,\beta)$.
\end{cor}

\begin{defn}
 Let $T$ be a Kupershmidt operator on a  \textsf{Hom-MalcevRep} pair  $( A,[-,-],\alpha,\varrho,\beta)$. Two one-parameter
infinitesimal deformations $T^1_t=T+t\mathcal{T}_1$ and
$T^2_t=T+t\mathcal{T}_2$ are said to be  equivalent if there exists
an $x\in A$ such that $\alpha(x)=x$ and the pair  $(Id_ A+t ad_x,Id_V+t\varrho(x))$ is a
homomorphism   from $T^2_t$ to $T^1_t$.
\end{defn}
Let $(Id_ A+t ad_x,Id_V+t\varrho(x))$ be a homomorphism from
$T^2_t$ to $T^1_t$. Then $Id_ A+tad_x$ is a  Hom-Malcev algebra
endomorphism of $ A$. Thus, we have
$$
(Id_ A+t ad_x)[y,z]=[(Id_ A+t ad_x)(y),(Id_ A+t ad_x)(z)], \;\forall \  y,z\in  A,
$$
which implies that $x$ satisfies
\begin{equation}
 [[x,y],[x,z]]=0,\quad \forall \  y,z\in A.
 \label{eq:Nij1}
 \end{equation}
Then by \eqref{defi:isocon1}, we get
$$
(T+t\mathcal{T}_1)(Id_V+t\varrho(x))(a)=(Id_ A+t ad_x)(T+t\mathcal{T}_2)(a),\quad\forall \  a\in V,
$$
which implies
\begin{eqnarray}
 (\mathcal{T}_2-\mathcal{T}_1)(a)&=&T\varrho(x)(a)+[T(a),x],\label{eq:deforiso1} \\
  \mathcal{T}_1\varrho(x)(a)&=&[x,\mathcal{T}_2(a)], \; \forall \  a\in V.
  \label{eq:deforiso2}
\end{eqnarray}
Since $\alpha(x)=x$, by   \eqref{defi:isocon2}, we have
$$\beta\circ(Id_V+t\varrho(x))=(Id_{V}+t\varrho(x))\circ\beta.$$
Finally, \eqref{defi:isocon3} gives
$$
(Id_V+t\varrho(x))\varrho(y)(a)=\varrho((Id_ A+t ad_x)(y))(Id_V+t\varrho(x))(a),\quad \forall \  y\in A, a\in V,
$$
which implies that $x$ satisfies
\begin{equation}
  \varrho([x,y])\varrho(x)=0,\quad\forall \  y\in A.\label{eq:Nij2}
\end{equation}
\begin{defn}
Suppose that $T: V\rightarrow A$ is a Kupershmidt operator on a Hom-MalcevRep pair  $( A,[-,-],\alpha,\varrho,\beta)$. A linear deformation $T_t:T+t\mathcal{T}$ is said to be trivial if it is equivalent to the deformation $T_0=T$.
\end{defn}

    \begin{defn}
Let $T$ be a Kupershmidt operator on a  \textsf{Hom-MalcevRep} pair  $( A,[-,-],\alpha,\varrho,\beta)$. An element $x\in A$ is called a {\bf Nijenhuis element} associated to $T$ if $x$ satisfies $\alpha(x)=x$, the identities ~\eqref{eq:Nij1}, \eqref{eq:Nij2} and the identity
      \begin{eqnarray}
        ~[x,[T(a),x]+T\varrho(x)(a)]=0,\quad \forall \  a\in V.
         \label{eq:Nij3}
        \end{eqnarray}
   Denote by $Nij(T)$ the set of Nijenhuis elements associated to a Kupershmidt operator $T$.
    \end{defn}

By \eqref{eq:Nij1}-\eqref{eq:Nij2}, it is obvious that a
trivial one-parameter infinitesimal deformation gives rise to a
Nijenhuis element. The following result is in close analogue to
the fact that the differential of a Nijenhuis operator on a  Hom-Malcev
algebra generates a trivial one-parameter infinitesimal
deformation of the  Hom-Malcev algebra, justifying the notion of
Nijenhuis elements.

  \begin{thm}\label{thm:trivial}
   Let $T$ be a Kupershmidt operator on a  \textsf{Hom-MalcevRep} pair  $( A,[-,-],\alpha,\varrho,\beta)$. Then for any  $x\in Nij(T)$, $T_t:=T+t \mathcal{T}$ with $\mathcal{T}(a)=T\varrho(x)(a)+[T(a),x]$, for all $a\in V$, is a trivial one-parameter infinitesimal  deformation of the  Kupershmidt operator $T$.
\end{thm}

\begin{proof}
First $\mathcal{T}$ is closed since $\mathcal{T}(a)=T\varrho(x)(a)+[T(a),x]$.
To show that $\mathcal{T}$ generates a trivial one-parameter infinitesimal deformation of
the  Kupershmidt operator $T$, we only need to verify that
\eqref{commuting condition} and \eqref{eq:deform2} hold. By definition of $\mathcal{T}$ and using the condition $\alpha(x)=x$ we have,
 for all $a\in V$,
\begin{align*}
\alpha\circ \mathcal{T}(a)&=\alpha( T\varrho(x)(a)+[T(a),x])=T(\beta(\varrho(x)(a))+\alpha([T(a),x]))\\&=T(\varrho(\alpha(x))(\beta(a)))+[\alpha(T(a)),\alpha(x)]))=T(\varrho(x)(\beta(a)))+[T(\beta(a)),x]))\\&=\mathcal{T}(\beta(a)).
\end{align*}
Then  $\mathcal{T}$ satisfies \eqref{commuting condition}.
   Moreover, according to   \eqref{eq:Nij1}, we have,
for any $a,b\in V$,
\begin{align*}
      &[\mathcal{T} (a),\mathcal{T} (b)]-\mathcal{T}(\varrho(\mathcal{T}(a) )(b)-\varrho(\mathcal{T} (b))(a))\\
      &= [[T(a),x],[T(b),x]]+[[T(a),x],T\varrho(x)(b)]+[T\varrho(x)(a),[T(b),x]]+[T\varrho(x)(a),T\varrho(x)(b)]\\
      &-[T\varrho([T(a),x])(b),x]-[T\varrho(T\varrho(x)(a))(b),x]+[T\varrho([T(b),x])(a),x]+[T\varrho(T\varrho(x)(b))(a),x]\\
      &-T\varrho(x)\varrho([T(a),x])(b)-T\varrho(x)\varrho(T\varrho(x)(a))(b)+T\varrho(x)\varrho([T(b),x])(a)+T\varrho(x)\varrho(T\varrho(x)(b))(a)\\
      &=[[T(a),x],T\varrho(x)(b)]+[T\varrho(x)(a),[T(b),x]]+\underline{T\varrho(T\varrho(x)(a))\varrho(x)(b)}\underbrace{-T\varrho(T\varrho(x)(b))\varrho(x)(a)}\\
      &-[T\varrho([T(a),x])(b),x]-[T\varrho(T\varrho(x)(a))(b),x]+[T\varrho([T(b),x])(a),x]+[T\varrho(T\varrho(x)(b))(a),x]\\
      &\underline{-T\varrho(x)\varrho([T(a),x])(b)-T\varrho(x)\varrho(T\varrho(x)(a))(b)}\underbrace{+T\varrho(x)\varrho([T(b),x])(a)+T\varrho(x)\varrho(T\varrho(x)(b))(a)}.
    \end{align*}
 The under-braced terms add to zero  using \eqref{eq:Nij2} and \eqref{eq:Nij3} and  the underlined terms add to zero by the same computation.
For the other terms, by \eqref{eq:Nij1} and \eqref{eq:Nij3}, we have
\begin{eqnarray*}
  &&[[T(a),x],T\varrho(x)(b)]-[T\varrho([T(a),x])(b),x]+[T\varrho(T\varrho(x)(b))(a),x]\\
  &=&[T(a),[x,T\varrho(x)(b)]]+[[T(a),T\varrho(x)(b)],x]-[T\varrho([T(a),x])(b),x]+[T\varrho(T\varrho(x)(b))(a),x]\\
  &=&-[T(a),[x,[T(b),x]]]+[T\varrho(T(a))\varrho(x)(b),x]-[T\varrho([T(a),x])(b),x]\\
  &=&-[x,[T(a),[T(b),x]]]+[T\varrho(x)\varrho(T(a))(b),x]\\
  &=&-[x,[T(a),[T(b),x]]]+[x,[T\varrho(T(a))(b),x]].
\end{eqnarray*}
Similarly, we have
\begin{eqnarray*}
&&[T\varrho(x)(a),[T(b),x]]-[T\varrho(T\varrho(x)(a))(b),x]+[T\varrho([T(b),x])(a),x]\\
 &=&[x,[T(b),[T(a),x]]]-[x,[T\varrho(T(b))(a),x]].
\end{eqnarray*}
Thus
 \begin{eqnarray*}
      &&[\mathcal{T} (a),\mathcal{T} (b)]-\mathcal{T}(\varrho(\mathcal{T} (a))(b)-\varrho(\mathcal{T} (b))(a))=0,
\end{eqnarray*}
Then,  $\mathcal{T}$ generates  a   one-parameter infinitesimal  deformation of $T$.

Further, since $x$ is a Nijenhuis element, it is straightforward that $(Id_ A+t ad_x,Id_V+t\varrho(x))$ gives the
desired homomorphism between $T_t$ and $T$. Thus, the deformation is trivial.
\end{proof}

Now we introduce the notion of a Nijenhuis operator on a Hom-pre-Malcev algebra.

\begin{defn}
  A linear map $N:A\longrightarrow A$ on a Hom-pre-Malcev algebra $(A,\cdot,\alpha) $ is called a {\bf Nijenhuis operator} if
   \begin{equation}
    N\circ\alpha=\alpha\circ N,
  \end{equation}
  \begin{equation}
    N(x)\cdot N(y)=N\big(N(x)\cdot y+ x\cdot  N(y)-N(x\cdot y)\big),\quad\forall \  x,y\in A.
  \end{equation}
   \label{df:nijenhuisoponhpm}
\end{defn}
\begin{prop}
If $N$ is a Nijenhuis operator on a Hom-pre-Malcev algebra $(A, \cdot, \alpha)$, then $N$ is a
Nijenhuis operator on the sub-adjacent Hom-Malcev algebra $(A,[-,-],\alpha)$.
\end{prop}
\begin{proof}
Since $N$ is a Nijenhuis operator on $(A, \cdot, \alpha)$, we have $ N\circ\alpha=\alpha\circ N$. For all $x, y\in A$, by Proposition~\ref{prop:HompreMalcevHomMalcevadmis} and Definition ~\ref{df:nijenhuisoponhpm} , we have
\begin{align*}
[N(x),N(y)] &= N(x) \cdot N(y) - N(y) \cdot N(x)\\
&= N\big(N(x) \cdot y + x \cdot N(y) - N(x\cdot y) - N(y) \cdot x - y \cdot N(x) + N(y \cdot x)\big)\\
&= N\big([N(x), y] + [x,N(y)] - N[x, y]\big).
\end{align*}
Thus, $N$ is a Nijenhuis operator on $(A,[-,-],\alpha)$.
\end{proof}
For its connection with a Nijenhuis element associated to
a Kupershmidt operator, we   have

\begin{prop}
Let $x\in A$ be a Nijenhuis element associated to a
 Kupershmidt operator $T$ on a  \textsf{Hom-MalcevRep} pair $(A,[-,-],\alpha,\varrho,\beta)$. Then $\varrho(x)$ is a Nijenhuis operator
on the associated Hom-pre-Malcev algebra $(V,\cdot, \beta)$.
    \end{prop}

\begin{proof}
For the proof, we just need to check, by
\eqref{eq:Nij2}, for all $a,b\in V$,
    \begin{align*}
     & \varrho(x)(\varrho(x)(a)\cdot b+a\cdot \varrho(x)(b)-\varrho(x)(a\cdot b))-\varrho(x)(a)\cdot\varrho(x)(b)\\
     &\quad =\varrho(x)\Big(\varrho(T\varrho(x)(a))(b)+\varrho(T(a))\varrho(x)(b)-\varrho(x)\varrho(T(a))(b)\Big)-\varrho(T\varrho(x)(a))\varrho(x)(b)\\
     &\quad =[\varrho(x),\varrho(T\varrho(x)(a))]+[\varrho(x),\varrho([T(a),x])](b)\\
     &\quad =\varrho([x,T\varrho(x)(a)+[T(a),x]])(b)
     =0. \qedhere
    \end{align*}
     \end{proof}


\end{document}